\documentclass[11pt]{amsart}
\usepackage[utf8]{inputenc}
\usepackage{amssymb,amsmath}
\usepackage{graphicx}
\usepackage{color}
\usepackage[hidelinks]{hyperref}
\usepackage{epstopdf}
\usepackage{tikz}
\usepackage{pgfplots}
\usepackage{enumitem}
\usepackage[leftcaption]{sidecap}
\usepackage{float}
\usepackage{algorithm}
\usepackage{setspace}
\usepackage{algpseudocode}
\usepackage{mathtools}
\usepackage{algorithm}
\usepackage{algpseudocode}
\usepackage{xfrac}
\usepackage{booktabs}
\usepackage{makecell}
\usepackage{todonotes}

\usepackage{accents}
\usepackage{multicol} 
\usepackage{soul}
\usepackage{subcaption}
\makeatletter
\newtheoremstyle{smartbrackets}
  {3pt} 
  {3pt} 
  {\itshape} 
  {} 
  {\normalfont} 
  {} 
  { } 
  {%
    \thmname{#1}~\thmnumber{#2}%
    \@ifnotempty{#3}{ [\thmnote{#3}]}%
  }
\makeatother

\theoremstyle{smartbrackets}
\newtheorem{lemma}{Lemma}[section]
\newtheorem{theorem}{Theorem}

\theoremstyle{definition}

\theoremstyle{plain}  

\theoremstyle{remark}
\newtheorem{remark}[theorem]{Remark}

\numberwithin{theorem}{section}
\numberwithin{lemma}{section}
\numberwithin{equation}{section}
\numberwithin{table}{section}
\numberwithin{figure}{section}
\def\V{V_h}
\def\Vo{V_{h,0}}

\def\O{\mathcal{O}}
\def\calO{\O}

\definecolor{blau}{RGB}{0, 51, 255}
\definecolor{hellblau}{RGB}{153, 204, 255}
\definecolor{hellrot}{RGB}{255, 0, 0}
\definecolor{firebrick}{RGB}{176, 34, 34} 
\definecolor{deep_pink}{RGB}{255, 20, 147} 
\definecolor{sky_blue}{RGB}{74, 112, 139}
\definecolor{slate_blue}{RGB}{71, 60, 139}
\definecolor{chartreuse}{RGB}{118, 238, 0}
\definecolor{chartreuseL}{RGB}{228, 255, 150}
\definecolor{light_blue}{RGB}{178, 223, 238}
\definecolor{dodge_blue}{RGB}{17, 78, 138}
\definecolor{code_backg}{RGB}{238, 216, 174}
\definecolor{myBlue1}{RGB}{101,149,239}  
\definecolor{myBlue2}{RGB}{113,104,238} 
\definecolor{myBlue3}{RGB}{30,144,255} 
\definecolor{myGreen1}{RGB}{154,204,50} 
\definecolor{myGreen2}{RGB}{69,169,0} 
\definecolor{myGreen3}{RGB}{154,205,50} 
\definecolor{myGreen4}{RGB}{105,139,34} 
\definecolor{myRed1}{RGB}{210,105,30} 
\definecolor{myRed2}{RGB}{165,42,42} 
\definecolor{myRed3}{RGB}{139,26,26} 
\definecolor{myLGray}{RGB}{225,225,225} 
\definecolor{mycolor1}{rgb}{0.00000,0.44700,0.74100}%
\definecolor{mycolor2}{rgb}{0.85000,0.32500,0.09800}%
\definecolor{mycolor3}{rgb}{0.92900,0.69400,0.12500}%
\definecolor{mycolor4}{rgb}{0.49400,0.18400,0.55600}%
\definecolor{mycolor5}{rgb}{0.46600,0.67400,0.18800}%
\definecolor{mycolor6}{rgb}{0.30100,0.74500,0.93300}%
\definecolor{mycolor7}{rgb}{0.63500,0.07800,0.18400}%
\allowdisplaybreaks

\begin{document}
\title{An implicit--explicit BDF--Galerkin scheme of second order for the nonlinear thermistor problem} 
\author[]{R.~Altmann$^\dagger$, A.~Moradi$^{\dagger}$ }
\address{${}^{\dagger}$ Institute of Analysis and Numerics, Otto von Guericke University Magdeburg, Universit\"atsplatz 2, 39106 Magdeburg, Germany}
\email{\{robert.altmann, afsaneh.moradi\}@ovgu.de}
\thanks{${^*}$ This project is funded by the Deutsche Forschungsgemeinschaft (DFG, German Research Foundation) through the project 446856041. \\ \indent This article will be published in \emph{IMA Journal of Numerical Analysis}.}
%
%
\date{\today}
%
%
\begin{abstract}
This paper proposes and analyzes an implicit--explicit BDF--Galerkin scheme of second order for the time-dependent nonlinear thermistor problem. For this, we combine the second-order backward differentiation formula with special extrapolation terms for time discretization with standard finite elements for spatial discretization. Unconditionally superclose and superconvergent error estimates are established, relying on two key techniques. First, a time-discrete system is introduced to decompose the error function into its temporal and spatial components. Second, superclose error estimates between the numerical solution and the interpolation of the time-discrete solution are employed to effectively handle the nonlinear coupling term. Finally, we present numerical examples that validate the theoretical findings, demonstrating the unconditional stability and the second-order accuracy of the proposed method.
\end{abstract}
%
%
\maketitle
\pagestyle{plain}
%
\noindent
{\tiny {\bf Keywords:} nonlinear thermistor problem, implicit--explicit, supercloseness, superconvergence, BDF--Galerkin method } 
%
%
\section{Introduction}\label{Sec.intro}
This paper deals with the numerical approximation of the time-dependent nonlinear thermistor problem
\begin{subequations}
\label{eq:thermistor}
    \begin{alignat}{3}
        \dot{u}\,-\,\Delta u\,&= \sigma(u)|\nabla \phi|^2&\quad \text{in $\Omega\times(0,T]$},\label{eq:thermistor:a}\\[1mm]
    -\,\nabla\cdot(\sigma(u)\nabla \phi)\,&= 0& \quad \text{in $\Omega\times(0,T]$},\label{eq:thermistor:b}\\
    u &=  0 &\quad \text{on $\partial \Omega \times (0,T]$},\label{eq:u:BC}\\
    \phi &=  g &\quad \text{on $\partial \Omega \times (0,T]$},\label{eq:phi:BC}\\
    u(\cdot ,0) &=  u^0 & \quad \text{in $\Omega$},\label{eq:IC}
\end{alignat}
\end{subequations}
also known as Joule heating problem~\cite{allegretto1992}. Therein, $\Omega\subseteq \mathbb{R}^2$ is a bounded two-dimensional domain with boundary $\partial\Omega$, $\dot{u}$ denotes the time derivative $\frac{\partial u}{\partial t}$, and $T<\infty$ is the time horizon. 
This coupled nonlinear parabolic--elliptic system describes the electric heating of a conducting body, 
in which the unknown $u$ is the temperature and $\phi$ represents the electric potential. Moreover, $ \sigma(u) $ is the temperature-dependent electrical conductivity, $ \sigma(u)|\nabla\phi|^2 $ acts as heat source, and $u^0$ and $g$ are given smooth functions. 

Over the past decades, the thermistor problem has been widely studied in both theoretical and numerical aspects due to its significance in electrical and thermal engineering applications. Theoretical investigations have established the existence, uniqueness, and regularity properties of solutions.
The existence of weak solutions was rigorously established in~\cite{cimatti1992} using a variational approach and techniques from functional analysis. Notably, the existence of solutions was also addressed in~\cite{allegretto1992,AntC94}, while~\cite{allegretto2002} extended the analysis to prove the existence of weak solutions over arbitrarily large time intervals. Furthermore, the existence and uniqueness of a $C^{\alpha}$-solution in three-dimensional space was established in~\cite{yuan1994a}. 

In parallel, significant efforts have been made to develop efficient numerical methods for approximating the solution of the thermistor problem~\eqref{eq:thermistor}. In~\cite{elliott1995}, optimal $L^2$-error estimates were derived for a semi-discrete and an implicit--explicit Euler scheme under the condition $ \tau\leq \calO(h^{d/6})$ with spatial dimension~$d \leq3$, where $\tau$ represents the time step size and $h$ denotes the spatial mesh size. In~\cite{zhao1994}, a mixed finite element method (FEM) is combined with an implicit--explicit Euler scheme to obtain optimal $L^2$-error estimates under the time step size condition $\tau\leq \calO(1/\sqrt{\log h^{-1}})$. Additionally, optimal $L^2$-error estimates were established in~\cite{akrivis2005} for fully discrete numerical methods, where the spatial discretization was performed using a high-order FEM and the time approximation employed a combination of rational implicit and explicit multi-step schemes. However, a condition $\tau \leq \calO(h^{3/2p})$, where $p$ is the order of time discretization, is required. In~\cite{shi2018,shi2019}, superconvergence was obtained under the time step size restriction $\tau \leq \calO(h^{1+\alpha})$ with $\alpha > 0$. With a focus on mixed boundary conditions on Lipschitz domains, further convergence results for a wide class of finite element schemes combined with the implicit Euler scheme were established in~\cite{JenMP20}.

The necessity of step size restrictions in previous works arises from the direct application of inverse inequalities using traditional techniques. To eliminate such restrictions, a splitting technique has been proposed and investigated in~\cite{li2012,li2014}. Unconditional optimal error estimates for fully discrete schemes involving BDF and Crank--Nicolson, using the same splitting technique, were considered in~\cite{gao2014,gao2016}. The splitting technique introduces a time-discrete (elliptic) system that decomposes the total error function into temporal and spatial components, making the spatial error independent of $\tau$. With this, time step size conditions can be removed by leveraging inverse inequalities and mathematical induction. 

Despite these advances, there has been limited work on deriving unconditionally superclose and superconvergent error estimates. Here, supercloseness refers to the improved accuracy between the numerical solution and an interpolation of the exact solution. On the other hand, superconvergence describes higher-than-expected convergence rates, often utilized in post-processing to enhance the overall accuracy. In~\cite{yang2023}, such error estimates for an implicit--explicit fully discrete scheme of first order were presented. The results demonstrated supercloseness of order $\mathcal{O}(h^2 + \tau)$ for the difference between the interpolated exact solution and the numerical solution in the $H^1$-norm.
Furthermore, an interpolation based post-processing yields superconvergence estimates of order $\mathcal{O}(h^2 + \tau)$, again in the $H^1$-norm.

Implicit--explicit methods of higher order in time are very rare. To the best of our knowledge, the only positive example is~\cite{gao2016}, where unconditionally optimal $L^2$-error estimates for a third-order implicit--explicit BDF scheme were derived. This, however, neither includes supercloseness nor superconvergence results. 
Building on this, we propose and analyze a fully discrete implicit--explicit BDF--Galerkin scheme of second order for the nonlinear thermistor problem~\eqref{eq:thermistor}, using a variation of the second-order BDF for time discretization and standard (bi)linear finite elements for spatial discretization. The scheme is explicit in the nonlinearity and further allows a decoupled computation of the two solution components. Similar as in \cite{AltM22}, this improves the computational efficiency compared to fully implicit methods. The analysis relies on the above mentioned splitting of the error components~\cite{li2012,li2014}, i.e., high-accuracy estimates for the finite element discretization and an error estimate of the numerical solution compared to the interpolation of the solution to the time-discrete system. In contrast to the approach in~\cite{gao2016}, we are able to prove supercloseness and superconvergence for the proposed scheme.

The remainder of the paper is organized as follows. In the forthcoming section, we introduce the necessary preliminaries and present the implicit--explicit BDF--Galerkin scheme. Sections \ref{Sec.Temporal} and \ref{Sec.Superclose} provide a detailed investigation of the error estimates in the temporal and spatial directions, respectively. Then, we collect the main results on superclose and superconvergent error estimates in Section~\ref{Sec:main}. In Section~\ref{Sec.Numeric}, numerical results are presented to confirm the validity of the presented analysis and identify limitations for non-smooth solutions. Finally, we conclude and give an outlook in Section~\ref{Sec:Conclusion}.
%
%
\section{Preliminaries}\label{Sec.Galerkin}
Before presenting the numerical scheme, we introduce some notation that will be used throughout the paper. For simplicity, we use $C$ to denote a generic positive constant, which is independent of the mesh size $h$ and the time step size $\tau$ but may vary from line to line in estimates. 

For an integer $k\geq 0$ and $1\leq p\leq \infty$, let $W^{k,p}(\Omega)$ denote the Sobolev space equipped with the norm $\|\cdot\|_{k,p}$ as introduced in~\cite{adams2003}. For $p=2$, we write $H^k(\Omega)\coloneqq W^{k,2}(\Omega)$ with corresponding norm
\[
    \|\cdot\|_{k}
    := \|\cdot\|_{k,2}.
\]
Moreover, we use $(\cdot,\cdot)$ as notation for the standard $L^2$-inner product. To shorten notation, we further introduce the spaces 
\[
    V\coloneqq H^1(\Omega), \qquad
    V_0\coloneqq H_0^1(\Omega).
\]
%
\subsection{Interpolation operator and important estimates}
Throughout this paper, $\mathcal{T}_h$ denotes a conforming mesh of the two-dimensional domain~$\Omega$ with mesh size $h$. The corresponding finite element space $\V$ is defined as
\begin{equation}\label{eq:space:Vh}
    \V \coloneqq \big\{ v \in C(\overline{\Omega}) : v|_K \in \mathcal{P}(K)\ \text{for all}\ K \in \mathcal{T}_h \big\},
\end{equation}
where 
\[
    \mathcal{P}(K) = 
    \begin{cases}
        \mathrm{span}\{1, x, y\}, & \text{if } K \text{ is a triangle}, \\
        \mathrm{span}\{1, x, y, xy\}, & \text{if } K \text{ is a rectangle }.
    \end{cases}
\]
Note that, in this work, we only consider rectangles aligned with the coordinate axes, leading to the standard bilinear finite element space~$Q_1$. For more general cases, isoparametric mappings are needed. 
We also make use of the subspace $\Vo\coloneqq \V\cap V_0$, which consists of discrete finite element functions that vanish on the boundary.

Let $I_h\colon W^{2,p}(\Omega)\rightarrow \V$ be the standard Lagrange interpolation operator associated to the mesh~$\mathcal{T}_h$. Then, the well-known estimates 
\begin{align}
    \|u-I_hu\|_{0,p} + h\, \|u-I_h u\|_{1,p}
    &\leq C\, h^2\, \|u\|_{2,p}, \label{ineq:Ih}\\
    \|I_hu\|_{1,p}
    &\leq C\, \|u\|_{1,p} \label{ineq:sta:Ih}
\end{align}
hold for all $u\in W^{2,p}(\Omega)$ and $1\leq p\leq \infty$; see~\cite{dupont1974} and \cite[Ch.~1]{thomee2007}. 
We summarize some further results that will be used throughout the paper.
\begin{lemma}(Interpolation estimates, {\cite[Ch.~2]{lin2007}})
\label{lem:sup:Ih}
Suppose $ u \in H^3(\Omega)\cap V_0$. Then for any $v_h \in \Vo $ it holds that 
\begin{align}
    \big(\nabla (u - I_h u), \nabla v_h\big) 
    &\leq C\, h^2\, \|u\|_3\, \|\nabla v_h\|_0, \label{ineq:lem:Ih}\\
    \big(\nabla (u - I_h u),  v_h\big) 
    &\leq C\, h^2\, \|u\|_3\, \| v_h\|_1. \label{ineq:lem:Ih2}
\end{align}
\end{lemma}
%
\begin{lemma}(Inverse estimate, {\cite[Lem.~4.5.3]{brenner2008}})
\label{lem:inverse}
Consider parameters $1\leq q\leq p\leq \infty$ and $m\geq 0$. Then it holds that
\begin{equation}
    \|v_h\|_{m,p} 
    \leq \,Ch^{(2/p-2/q)}\|v_h\|_{m,q}\qquad
    \text{for all $v_h\in \V$}.\label{ineq:invese:pq}
\end{equation}
\end{lemma}
\begin{lemma}(Discrete Gronwall inequality, \cite{heywood1990})
\label{lem:Gronwall}
Let $\tau$, $B$ as well as $a_k$, $b_k$, $c_k$, and $\gamma_k$ be non-negative numbers for $k\geq0$ such that
\[
    a_n+\tau\sum_{k=0}^nb_k\leq\tau \sum_{k=0}^n\gamma_ka_k\,+\,\tau \sum_{k=0}^nc_k+B,\qquad n\geq 0.
\]  
Suppose that $\tau\gamma_k<1$ for all $k$ and set $\sigma_k \coloneqq (1-\tau\gamma_k)^{-1}$. Then it follows that
\[
    a_n+\tau\sum_{k=0}^nb_k
    \leq \bigg(\tau\sum_{k=0}^nc_k+B\bigg)\exp\bigg(\tau\sum_{k=0}^n\gamma_k\sigma_k\bigg),\qquad n\geq 0.
\]
\end{lemma}
%
\subsection{Weak formulation and regularity assumptions}
The weak formulation of system~\eqref{eq:thermistor} seeks $u\colon(0,T]\rightarrow V_0$ and $\phi\colon(0,T]\rightarrow V$ such that
\begin{subequations}\label{eq:weakForm}
    \begin{alignat}{2}
        (\dot{u},\xi_{u})+(\nabla u,\nabla \xi_u)&=  (\sigma(u)|\nabla\phi|^2,\xi_u),\\
        (\sigma(u)\nabla\phi,\nabla\xi_{\phi})&= 0 
    \end{alignat}
\end{subequations}
for all $\xi_u,\, \xi_\phi\in V_0$ with boundary conditions \eqref{eq:u:BC} and \eqref{eq:phi:BC}. Following~\cite{yang2023}, we assume that $\sigma\in W^{2,\infty}(\mathbb{R})$ is Lipschitz continuous, i.e., there exist constants $L_{\sigma,p}>0$, $p<\infty$, such that
\[
    \| \sigma(u) - \sigma(v) \|_{0,p} 
    \leq L_{\sigma,p} \| u - v \|_{0,p}
\]
for all $u,\, v \in V$. Moreover, $\sigma$ satisfies the uniform bounds
\begin{equation}\label{ineq:sigma}
    0<k_1 \leq \sigma(s) 
    \leq k_2 \qquad \text{for all } s \in \mathbb{R}
\end{equation}
for some constants $k_1, k_2 > 0$. 
Due to $\sigma \in W^{2,\infty}(\mathbb{R})$, also the first derivative is uniformly bounded and Lipschitz continuous. Hence, there exists a constant $L_{\sigma'}>0$ such that
\[
    \| \sigma'(u) - \sigma'(v) \|_{0} 
    \leq L_{\sigma'} \| u - v \|_{0}
\]
for all $u,\, v \in V$ and we have 
\[
|\sigma'(s)| \leq {k}_3 \qquad \text{for all } s \in \mathbb{R}.
\] 
As in \cite{gao2014,li2014}, we assume that the solutions to the initial boundary value problem \eqref{eq:thermistor} exist and satisfy the regularity conditions 
\begin{equation}\label{eq:regularity}
    \left\{
    \begin{array}{l}
         \|u\|_{L^{\infty}(H^3(\Omega))} 
         + \|\dot{u}\|_{L^{\infty}(W^{1,\infty}(\Omega))} 
         + \|\ddot{u}\|_{L^{\infty}(H^1(\Omega))} 
         + \|\dddot{u}\|_{L^{\infty}(L^2(\Omega))} \leq C, \\
         \|\phi\|_{L^{\infty}(H^3(\Omega))} 
         + \|\dot{\phi}\|_{L^{\infty}(H^3(\Omega))} 
         + \|\ddot{\phi}\|_{L^{\infty}(W^{2,4}(\Omega))} 
         + \|g\|_{L^{\infty}(W^{2,4}(\Omega))} \leq C.
    \end{array}
    \right.
\end{equation}
\begin{remark}
In order to meet these strong regularity assumptions, one would expect at least a domain with $C^{2,1}$-boundary, initial data $u^0 \in H^3(\Omega)$, boundary data $g\in H^{5/2}(\partial\Omega)$, as well as certain compatibility conditions between $g$ and $\phi(0)$. 
In the numerical experiments of Section~\ref{Sec.Numeric}, however, we will observe that second-order convergence can already be reached for less regular solutions. 
\end{remark}
%
\subsection{Implicit--explicit BDF--Galerkin scheme}
To define the fully discrete scheme, let $\{t^n\}_{n=0}^N$ be a partition in time direction with final time $T$, time step size~$\tau=T/N$, and $t^n=n\tau$. For any sequence of functions $\{f^n\}_{n=0}^N$, we define the second-order BDF operator by
\begin{equation}\label{eq:BDF2}
    D_{\tau}f^{n} 
    = \frac{1}{2\tau}\, \big(3f^{n}-4f^{n-1}+f^{n-2}\big), \qquad 2\leq n\leq N.
\end{equation}
\begin{lemma}(Telescope formula for $D_\tau$, \cite{liu2013})
\label{lem:tel:BDF}
With the definition of the BDF operator $D_\tau$ in \eqref{eq:BDF2},  the following telescope formula holds
    \begin{align*}
        \big(D_\tau f^{n},f^{n}\big) 
        = \frac{1}{4\tau}\, \Big(\|f^{n}\|_0^2 -  \|f^{n-1}\|_0^2\Big) + \frac{1}{4\tau}\, \Big( \|2f^{n}-f^{n-1}\|_0^2 -  \|2f^{n-1}-f^{n-2}\|_0^2\Big).
    \end{align*}
    \end{lemma}
With the aforementioned notation, we introduce the implicit--explicit second-order BDF--Galerkin FEM for the nonlinear thermistor problem~\eqref{eq:thermistor}. Given $U_h^0,U_h^1$, 
we seek $ U_{h}^{n} \in \Vo $ and $ \Phi_{h}^{n} \in \V $ such that  
\begin{subequations}
  \label{eq:BDF:Galerkin}
    \begin{align}
        (D_{\tau}U_h^{n},\xi_u) + (\nabla U_h^{n},\nabla \xi_u)& =  \left((2\sigma(U_h^{n-1})-\sigma(U_h^{n-2}))|\nabla\Phi_{h}^{n}|^2,\xi_u\right),  \label{eq:BDF:Galerkin:a} \\
        \left((2\sigma(U_h^{n-1})-\sigma(U_h^{n-2}))\nabla\Phi_h^{n},\nabla\xi_\phi\right)& =0 \label{eq:BDF:Galerkin:b}
    \end{align}
\end{subequations}   
for all $\xi_u,\, \xi_\phi\in \Vo$. As boundary condition for the electric potential, we set $\Phi_{h}^{n}|_{\partial\Omega} = I_h g^{n}$, where $I_h$ denotes again the interpolation operator but now restricted to the boundary.

Scheme \eqref{eq:BDF:Galerkin} is decoupled in the sense that one can solve \eqref{eq:BDF:Galerkin:b} for $ \Phi_{h}^{n} $ first and then \eqref{eq:BDF:Galerkin:a} for $U_{h}^{n}$ in each time step. Moreover, both equations are linear. To initialize the procedure, we need $U_h^0$ and $U_h^1$. There are multiple ways to compute such starting values while preserving the second-order accuracy of the BDF--Galerkin FEM \eqref{eq:BDF:Galerkin}. 
A commonly used procedure to compute $U_h^1$ and $\Phi_h^1$ is to apply one step of the well-studied implicit--explicit Euler Galerkin FEM~\cite{gao2016,yang2023}. This means that with $U_h^0=I_h u^0$ we solve 
\begin{subequations}
\label{eq:int:Uh11_Phi11}
\begin{alignat}{2}
    \big({U}_h^1-U_h^0, \xi_u \big) +  \tau\big(\nabla {U}_h^1, \nabla \xi_u \big) 
    &= \tau\big( \sigma(U_h^0) |\nabla {\Phi}_h^1|^2, \xi_u \big), \label{eq:int:Uh11} \\
    \big( \sigma(U_h^0) \nabla {\Phi}_h^1, \nabla \xi_\phi \big) &= 0 \label{eq:int:Phih11}
\end{alignat}
\end{subequations}
for all $ \xi_u,\, \xi_\phi \in \Vo $. This system is again decoupled and the two equations can be solved sequentially. Although the method is only first-order accurate, is suffices to establish a second-order error estimate as we will show in the upcoming sections. 

In the next two sections, we analyze the temporal and spatial error functions separately. This well then be combined in Section~\ref{Sec:main}, where we present the main results.
%
%
\section{Temporal error estimates}\label{Sec.Temporal}
In this section, we derive temporal error estimates for the solution of the time-discrete system obtained using an implicit--explicit BDF method, based on the error splitting technique introduced in~\cite{li2012}. To this end, we define 
$U^{n}\in V_0$ 
and $\Phi^{n}\in V$  
for $n=2,\ldots,N$, as the solution of the elliptic problem
\begin{subequations}\label{eq:timeDis:str}
\begin{alignat}{2}
    D_\tau U^n -\Delta U^n &=  \big(2\sigma(U^{n-1})-\sigma(U^{n-2}) \big)\, |\nabla \Phi^{n}|^2,\label{eq:timeDis:str:a}\\
    -\nabla\cdot \left((2\sigma(U^{n-1})-\sigma(U^{n-2}))\nabla\Phi^{n}\right)&= 0
\end{alignat}
\end{subequations}
with Dirichlet boundary conditions
\begin{equation}\label{eq:bc}
    U^{n}=0,\quad 
    \Phi^{n}=g^n\qquad \text{ on $\partial \Omega$}
\end{equation}
and given initial data $U^0, U^1$.
The weak formulation of system \eqref{eq:timeDis:str} is given by 
    \begin{subequations}\label{eq:timeDis}
        \begin{alignat}{2}
        (D_{\tau}U^{n},\xi_u) \, +\, (\nabla U^{n},\nabla \xi_u)\,&= \left((2\sigma(U^{n-1})-\sigma(U^{n-2}))\, |\nabla \Phi^{n}|^2,\xi_u\right),\label{eq:timeDis:a}\\
        \left((2\sigma(U^{n-1})-\sigma(U^{n-2}))\nabla\Phi^{n},\nabla\xi_\phi\right) &= 0\label{eq:timeDis:b}
    \end{alignat}
    \end{subequations}
for all $\xi_u,\, \xi_\phi\in V_0$. The required initial values $U^1$ and $\Phi^1$ are obtained as the solution of  
\begin{subequations}\label{eq:weak:starting}
        \begin{alignat}{2}
        \left({{U}^1-U^0},\xi_u\right)+\tau\big(\nabla U^1,\nabla\xi_u\big)&= \tau\big(\sigma(U^0)|\nabla{\Phi}^1|^2,\xi_u\big),\\
        \big(\sigma(U^0)\nabla{\Phi}^1,\nabla\xi_\phi\big)&= 0
    \end{alignat}
\end{subequations}
for all $\xi_u,\, \xi_\phi\in V_0$ with $U^0=u^0$ and boundary conditions \eqref{eq:bc}. Throughout this paper, we denote point evaluations of the exact solution by 
\[
    u^n\coloneqq u(\cdot,t^n),\qquad 
    \phi^n\coloneqq \phi(\cdot,t^n). 
\]
Moreover, we assume without further mentioning that system~\eqref{eq:thermistor} has a unique solution $(u,\phi)$ satisfying \eqref{eq:regularity}.
\begin{lemma}(Estimate of starting approximations)
\label{lem:ini:est:time}
For sufficiently small $\tau$, the unique solution $(U^1,\Phi^1)$ of the elliptic system \eqref{eq:weak:starting} satisfies the error estimate
\begin{equation}\label{ini:est:time}
    \|U^1-u^1\|_{1}\, +\, \tau^{1/2}\,\|U^1-u^1\|_2\, +\, \|\Phi^1-\phi^1\|_2 
    \leq C\, \tau^2.
\end{equation}
\end{lemma}
\begin{proof}
This lemma can be proven by following the analysis of~\cite[Th.~2]{yang2023} in which the nonlinear thermistor equations are solved by the implicit--explicit Euler scheme.
\end{proof}
\begin{theorem}(Temporal error estimate)
\label{thm:TemError} 
Consider $U^0 = u^0$ and $(U^1,\Phi^1)$ from~\eqref{eq:weak:starting}. Then, for sufficiently small $\tau$, the unique solution $(U^n,\Phi^n)$ of the elliptic system \eqref{eq:timeDis} satisfies the error estimate
\begin{equation}\label{ineq:thm:error1}
    \|U^n-u^n\|_{1}\,+\, \tau^{1/2}\,\|U^n-u^n\|_2\,+\, \|\Phi^n-\phi^n\|_2
    \leq C\, \tau^2
\end{equation}
as well as the stability estimate
\begin{equation}\label{ineq:thm:error2}
     \|U^n\|_2\, +\, \tau\,\sum_{k=2}^n\|D_{\tau}U^k\|_{2}^2\, +\, \|\Phi^n\|_{2,4}
     \leq C
\end{equation}
for all $n=2,\ldots,N$. 
\end{theorem}
\begin{proof}
We define the temporal errors as $e_\tau^n\coloneqq U^n-u^n\in H^3(\Omega)\cap V_0$ and $\eta_\tau^n\coloneqq\Phi^n-\phi^n\in H^3(\Omega)\cap V_0$. By the weak formulation \eqref{eq:weakForm} evaluated at time $t^n$ and the elliptic system \eqref{eq:timeDis}, the temporal error functions $e_\tau^n$ and $\eta_\tau^n$ satisfy
    \begin{subequations}
    \label{eq:errors}
        \begin{alignat}{2}
            (D_{\tau}e_\tau^{n},\xi_u)\,+\,(\nabla e_\tau^{n},\nabla \xi_u)=& \,2\left(\sigma(U^{n-1})|\nabla\Phi^{n}|^2-\sigma(u^{n-1})|\nabla\phi^{n}|^2,\xi_u\right)\,\nonumber\\
            & -\left(\sigma(U^{n-2})|\nabla\Phi^{n}|^2-\sigma(u^{n-2})|\nabla\phi^{n}|^2,\xi_u\right)\,-\left(R_u^{n},\xi_u\right),\label{eq:errors:a}\\ \big((2\sigma(U^{n-1})-\sigma(U^{n-2}))\nabla\eta_\tau^{n},\nabla\xi_\phi\big)=&-2\left((\sigma(U^{n-1})-\sigma(u^{n-1}))\nabla\phi^{n},\nabla\xi_\phi\right) \nonumber\\
            &  + \left((\sigma(U^{n-2})-\sigma(u^{n-2}))\nabla\phi^{n},\nabla\xi_\phi\right)+ ( R_\phi^{n},\nabla\xi_\phi)\label{eq:errors:b}
        \end{alignat}
    \end{subequations}
    for all $\xi_u,\, \xi_\phi\in V_0$ and $n=2,\ldots,N$. Therein, the truncation errors $R_{u}^{n}$ and $R_\phi^n$ are given by
\[
    R_{u}^{n} 
    = D_{\tau}u^{n}-\dot{u}^{n} + \big(\sigma(u^{n})-2\sigma(u^{n-1})+\sigma(u^{n-2})\big)\, |\nabla\phi^{n}|^2 
\]
and 
\[
    R_\phi^n 
    = \left(\sigma(u^n) - 2\sigma(u^{n-1}) + \sigma(u^{n-2})\right)\nabla\phi^n.
\]
By a Taylor expansion in time and using the regularity assumptions \eqref{eq:regularity}, it is straightforward to verify that the truncation errors satisfy the estimates
\begin{equation} \label{ineq:R:estimates}
    \|R_u^n\|_0 
    \leq C\, \tau^2, \qquad 
    \|R_\phi^n\|_0 
    \leq C\, \tau^2,\qquad 
    \|\nabla \cdot R_\phi^n\|_0 
    \leq C\, \tau^2.
\end{equation}

We now prove the estimate 
\begin{equation}\label{eq:priEst}
    \|e_\tau^n\|_2
    \leq C\,\tau^{3/2},\qquad n =0,1,\ldots,N,
\end{equation}
using mathematical induction. The case $n=0$ is trivial, since $e_\tau^{0}=U^0-u^0=0$. From \eqref{ini:est:time}, estimate \eqref{eq:priEst} holds for $n=1$. 
Now assume that \eqref{eq:priEst} holds up to index $n-1$ and we prove that the estimate remains true for $n$. In this case, we have 
\begin{align}
    \|U^{n-1}-U^{n-2}\|_{\infty}&\leq  \|e_\tau^{n-1}\|_{\infty}+\|e_\tau^{n-2}\|_{\infty}   + \tau\, \max_{s\in (t^{n-2},t^{n-1})} \|\dot{u}(s)\|_\infty
    \leq C\, \tau.\label{ineq:diff}
\end{align}
Setting $\xi_\phi = \eta_\tau^{n}$ in \eqref{eq:errors:b}, we have 
\begin{align*}
       \left((2\sigma(U^{n-1})-\sigma(U^{n-2}))\nabla\eta_\tau^{n},\nabla\eta_{\tau}^{n}\right)&=  \,-2\left((\sigma(U^{n-1})-\sigma(u^{n-1}))\nabla\phi^{n},\nabla\eta_{\tau}^{n}\right) \\
      & \quad +\left((\sigma(U^{n-2})-\sigma(u^{n-2}))\nabla\phi^n,\nabla\eta_\tau^n\right) +  (R_{\phi}^n,\nabla\eta_\tau^n)\\
     &\eqqcolon A_1 + A_2 + A_3.
    \end{align*}
Using \eqref{ineq:sigma}, \eqref{ineq:diff}, and the regularity assumptions \eqref{eq:regularity}, we deduce 
\begin{equation}\label{eq:k*} 
    k_*
    := k_1 - C k_3 \tau 
    \le 2\sigma(U^{n-1}) - \sigma(U^{n-2}),
\end{equation}
where $k_* > 0$ for sufficiently small $\tau$. Therefore, we have
\begin{align*}
      k_* \|\nabla\eta_\tau^{n}\|_0^2
      \leq \left((2\sigma(U^{n-1})-\sigma(U^{n-2}))\nabla\eta_\tau^{n},\nabla\eta_{\tau}^{n}\right).
\end{align*}
Applying the Cauchy--Schwarz inequality to the right-hand side, we obtain
\begin{align*} 
    A_1 
    \leq 2\,\|\sigma(U^{n-1})-\sigma(u^{n-1})\|_0\,\|\nabla\phi^{n}\|_{\infty}\,\|\nabla\eta_\tau^{n}\|_0 
    \leq C\,\|\nabla e_\tau^{n-1}\|_0\,\|\nabla\eta_\tau^{n}\|_0,  
\end{align*} 
where Poincaré’s inequality has been used in the last step. Similarly, one shows
\[A_2\leq C\,\|\nabla e_\tau^{n-2}\|_0\,\|\nabla\eta_\tau^{n}\|_0.\]
Moreover, with \eqref{ineq:R:estimates} we have
\[A_3\leq \|R_\phi^n\|_0\|\nabla\eta_\tau^n\|_0 \leq C \tau^2\|\nabla\eta_\tau^n\|_0.\]
As a result, we obtain  
\begin{equation}\label{ineq:eta0}
    \|\nabla\eta_\tau^{n}\|_0
    \leq C\left(\|\nabla e_\tau^{n-1}\|_0+\|\nabla e_\tau^{n-2}\|_0 + \tau^2\right).
\end{equation}
Now, multiplying \eqref{eq:timeDis:str:a} by $ -\Delta e_\tau^{n}$, integrating by parts, and subtracting the result from the weak formulation~\eqref{eq:weakForm}, we obtain
\begin{align}
    \big(D_\tau\nabla e_\tau^{n},\nabla e_\tau^{n}\big)+\|\Delta e_\tau^{n}\|_0^2
    &= 2\,\big(\sigma(U^{n-1})|\nabla \Phi^{n}|^2-\sigma(u^{n-1})|\nabla \phi^{n}|^2,-\Delta e_\tau^{n}\big) \nonumber\\
    &\qquad-\left(\sigma(U^{n-2})|\nabla \Phi^{n}|^2-\sigma(u^{n-2})|\nabla\phi^n|^2,-\Delta e_\tau^{n}\right) + \big(R_u^{n},\Delta e_\tau^{n}\big) \nonumber\\
    &\eqqcolon B_1 + B_2 + B_3.\label{eq:est}
\end{align}
    For $B_1$,  we have
    \begin{align*}
        B_1 
        &=  2\left( (\sigma(U^{n-1})-\sigma(u^{n-1}))|\nabla \phi^n|^2 + \sigma(U^{n-1})\nabla \eta_\tau^n(\nabla \eta_\tau^n+2\nabla \phi^n),-\Delta e_\tau^n\right) \nonumber\\
        &\leq C\left( 
         L_{\sigma,2}\|e_\tau^{n-1}\|_0\|\nabla \phi^n\|_{\infty}+2k_2\|\nabla\phi^n\|_{\infty}\|\nabla \eta_\tau^n\|_0+k_2\|\nabla\eta_\tau^n\|_{0,4}^2
        \right)\|\Delta e_\tau^n\|_0\nonumber\\
        &\leq  C\left(\|e_\tau^{n-1}\|_0+\|\nabla \eta_\tau^{n}\|_0+\|\nabla \eta_\tau^{n}\|_{0,4}^2\right)\|\Delta e_\tau^{n}\|_0\nonumber\\
        &\leq  C\left(\|\nabla e_\tau^{n-1}\|_0^2+\|\nabla e_\tau^{n-2}\|_0^2+\|\nabla \eta_\tau^{n}\|_{0,4}^4+\tau^4\right)+\tfrac14\, \|\Delta e_\tau^{n}\|_0^2,
    \end{align*}
    where \eqref{ineq:eta0} and  Young's inequality have been used in the last step. For $B_2$, we obtain the same estimate. 
    With \eqref{ineq:R:estimates}, it holds
    \begin{equation*}
        B_3\leq \|R_u^{n}\|_0\,\|\Delta e_\tau^{n}\|_0\leq C\tau^2\|\Delta e_\tau^{n}\|_0\leq C\tau^4+\tfrac14\, \|\Delta e_\tau^{n}\|_0^2. 
    \end{equation*}
    Substituting the above estimates for $B_1$, $B_2$, and $B_3$ into \eqref{eq:est}, we obtain
    \begin{align*}
        \tau\big(D_\tau\nabla e_\tau^{n},\nabla e_\tau^{n}\big)+\tfrac{\tau}{4}\, \|\Delta e_\tau^{n}\|_0^2
        \leq\, C \tau^5+C\tau\left( \|\nabla e_\tau^{n-1}\|_0^2 + \|\nabla e_\tau^{n-2}\|_0^2 +\|\nabla \eta_\tau^{n}\|_{0,4}^4\right).
    \end{align*}
    Summing up the last inequality and using the telescope formula for $D_\tau$ from Lemma \ref{lem:tel:BDF}, we obtain 
    \begin{equation}\label{eq:sum:e}
        \|\nabla e_\tau^{n}\|_0^2+\tau \sum_{\ell=1}^{n}\|\Delta e_\tau^{\ell}\|_0^2\leq C \tau^4 + C\tau\sum_{\ell=0}^{n-1}\|\nabla e_\tau^\ell\|_0^2 + C\tau\sum_{\ell=0}^{n-1}\|\nabla \eta_\tau^\ell\|_2^4. 
    \end{equation}
    On the other hand, \eqref{eq:errors:b} can be rewritten as 
    \begin{align}
        -\left(2\sigma(U^{n-1})-\sigma(U^{n-2})\right)\Delta \eta_\tau^{n} &= \, 2\nabla\cdot\big((\sigma(U^{n-1})-\sigma(u^{n-1}))\nabla\phi^{n}\big)\nonumber\\
        &\qquad -\nabla\cdot\left((\sigma(U^{n-2})-\sigma(u^{n-2}))\nabla\phi^{n}\right) -\nabla\cdot R_\phi^n\nonumber\\
        &\qquad + 2\sigma'(U^{n-1})\nabla U^{n-1}\cdot \nabla \eta_\tau^{n}-  \sigma'(U^{n-2})\nabla U^{n-2}\cdot \nabla \eta_\tau^{n}\nonumber\\
        &\eqqcolon D_1 + D_2 + D_3 + D_4 + D_5.\label{eq:Delat:eta}
    \end{align}
Using the regularity assumptions~\eqref{eq:regularity}, one can show with Poincaré’s inequality $D_1 \le C\,\|\nabla e_\tau^{n-1}\|_0$ 
%
as well as $D_2\leq C\,\|\nabla e_\tau^{n-2}\|_0$. 
By \eqref{ineq:R:estimates}, it directly follows that $D_3 \leq C\tau^2$. By Hölder’s inequality, we further obtain 
\begin{align*}
    D_4
    &\leq 2k_3\left(\|\nabla e_\tau^{n-1}\cdot\nabla \eta_\tau^n\|_0 + \|\nabla u^{n-1}\cdot \nabla \eta_\tau^n\|_0\right)\nonumber\\
    &\leq C\, \|\nabla e_\tau^{n-1}\|_{0,4}\|\nabla \eta_\tau^n\|_{0,4}
    \leq C\, \|e_\tau^{n-1}\|_2\|\nabla\eta_\tau^{n}\|_2
\end{align*}
and, similarly, $D_5 \le C\,\|e_\tau^{n-2}\|_2\|\nabla\eta_\tau^{n}\|_2$.
%
By applying elliptic regularity theory~\cite{chen1998} to~\eqref{eq:Delat:eta} and substituting the estimates for $D_1,\ldots,D_5$, we obtain
\begin{align}
    k_* \|\eta_\tau^{n}\|_2
    &\leq C\, \big( \big(\|e_\tau^{n-1}\|_2 + \|e_\tau^{n-2}\|_2\big)\,\|\eta_\tau^{n}\|_2 + \|\nabla e_\tau^{n-1}\|_0 + \|\nabla e_\tau^{n-2}\|_0 +\tau^2\big)\nonumber\\
    &\leq C\, \big(\tau^{3/2}\|\eta_\tau^{n}\|_2 + \|\nabla e_\tau^{n-1}\|_0  + \|\nabla e_\tau^{n-2}\|_0  +\tau^2\big).\nonumber
\end{align}
Hence for sufficiently small $\tau$, it follows that 
    \begin{equation}
        \|\eta_\tau^n\|_2
        \leq C\, \big(\|\nabla e_\tau^{n-1}\|_0 + \|\nabla e_\tau^{n-2}\|_0  +\tau^2\big).
    \label{eq:eta}
    \end{equation}
Therefore, with \eqref{eq:priEst} we have $\|\eta_\tau^{n}\|_2\leq 1$ for sufficiently small $\tau$.
    Substituting \eqref{eq:eta} and $\|\eta_\tau^{n}\|_2\leq 1$ into \eqref{eq:sum:e} and using the fact that $\|e_\tau^{n}\|_2\leq C\,\|\Delta e_{\tau}^{n}\|_0$, we obtain
    \begin{align*}
        \|\nabla e_\tau^{n}\|_0^2+\tau \sum_{\ell=1}^{n}\|e_\tau^\ell\|_2^2\leq C\tau^4+C\tau \sum_{\ell=0}^{n-1}\|\nabla e_\tau^\ell\|_0^2.
    \end{align*}
Applying Gronwall’s inequality from Lemma \ref{lem:Gronwall} then yields
\begin{align}\label{eq:nabla:e}
    \|\nabla e_\tau^{n}\|_0^2 + \tau\sum_{\ell=1}^{n} \|e_\tau^\ell\|_2^2  
    \leq \exp\left(\frac{CT}{1 - C\tau}\right) \tau^4
    \leq C\exp(2CT)\tau^4   
    \leq C\tau^4,  
\end{align}  
where we require $1 - C\tau \geq \frac{1}{2}$. It follows that \eqref{eq:priEst} is indeed satisfied for all $n\geq 0$. Moreover, from \eqref{eq:eta} and \eqref{eq:nabla:e}, we obtain  
\[
\|\eta_\tau^{n}\|_2 \leq C\tau^2.
\]  
Hence, the desired result \eqref{ineq:thm:error1} follows. Furthermore, by following the same approach as in~\cite{li2012, gao2016}, one can derive the stability estimate~\eqref{ineq:thm:error2}.
\end{proof}
%
%
\section{Spatial error estimates}\label{Sec.Superclose}
This section presents error estimates for the spatial error, where the convergence is measured with respect to suitable projections of the time-discrete solutions rather than the exact solution. 
\begin{lemma}(Estimates of starting approximations)
\label{lem:ini:est}
Let $(U^1, \Phi^1)$ denote the solution of~\eqref{eq:weak:starting}. Then, for sufficiently small $\tau$ and $h$, the fully discrete system~\eqref{eq:int:Uh11_Phi11} with $U_h^0 = I_h u^0$ has a unique solution $(U_h^1,\Phi_h^1)$, satisfying the error estimate
    \begin{equation}\label{ini:est:space}
        \|U_h^1-I_hU^1\|_1\, +\, \|\Phi_h^1-I_h\Phi^1\|_1
        \leq C\, \big( h^2+h\tau^2 \big). 
    \end{equation}
\end{lemma}
\begin{proof}
For the proof we refer to~\cite[Th.~3]{yang2023}, where the implicit--explicit Euler scheme was applied to the nonlinear thermistor problem.
\end{proof}
\begin{theorem}(Spatial error estimate)
\label{thm:err:spatial}
Consider $(U^n, \Phi^n)$ as the solution of~\eqref{eq:timeDis} with starting values given by~\eqref{eq:weak:starting}. Moreover, let $(U_h^n,\Phi_h^n)$ denote the fully discrete solution of system~\eqref{eq:BDF:Galerkin} for starting values $U_h^0 = I_h u^0$ and $(U_h^1,\Phi_h^1)$ from~\eqref{eq:int:Uh11_Phi11}. Then, for sufficiently small $\tau$ and $h$, we have
    \begin{equation}\label{ineq:err:full}
        \|U_h^n-I_hU^n\|_1+\|\Phi_h^n-I_h\Phi^n\|_1
        \leq C\, \big( h^2+h\tau^2 \big)
    \end{equation}
for all $n = 0,\ldots,N$. 
\end{theorem}
\begin{proof}
We define the errors as $e_h^n\coloneqq U_h^n-I_hU^n\in \Vo$ and $\eta_h^n\coloneqq\Phi_h^n-I_h\Phi^n\in \Vo$. As a first step, we prove 
\begin{equation}\label{ineq:priEst:spa}
    \|\nabla e_h^{n}\|_0 
    \leq C\, (h^2 + h \tau^2), \qquad n=0,1,\ldots,N, 
\end{equation}
using mathematical induction. The case $ n = 0 $ is trivial, since $ \|\nabla e_h^0\|_0 = \|\nabla(U_h^0 - I_h U^0)\|_0 = 0 $ by assumption. 
The case $n=1$ is covered by  \eqref{ini:est:space}. Now assume \eqref{ineq:priEst:spa} is valid up to index $n-1$ and we prove that the estimate remains true for $n$.
 
From \eqref{eq:BDF:Galerkin:b} and \eqref{eq:timeDis:b} considered with $\xi_\phi=\eta_h^{n}$ , the spatial error function $\eta_h^n$ satisfies
    \begin{alignat}{2}
        \left((2\sigma(U^{n-1})-\sigma(U^{n-2}))\nabla\eta_h^{n},\nabla \eta_h^{n}\right) &=  2\left((\sigma(U^{n-1})-\sigma(U_h^{n-1}))\nabla\eta_h^n,\nabla \eta_h^{n}\right)\nonumber\\
    &\qquad+ 2\left((\sigma(U^{n-1})-\sigma(U_h^{n-1}))\nabla I_h\Phi^{n},\nabla \eta_h^{n}\right)\nonumber\\
        &\qquad - \left((\sigma(U^{n-2})-\sigma(U_h^{n-2}))\nabla\eta_h^n,\nabla \eta_h^{n}\right)\nonumber\\
    &\qquad-\left((\sigma(U^{n-2})-\sigma(U_h^{n-2}))\nabla I_h\Phi^{n},\nabla \eta_h^{n}\right)\nonumber\\
        &\qquad + 2\left((\sigma(U^{n-1})-\sigma(u^{n-1}))\nabla(\Phi^{n}-I_h\Phi^{n}),\nabla\eta_h^{n}\right)\nonumber\\
        &\qquad - \left((\sigma(U^{n-2})-\sigma(u^{n-2}))\nabla(\Phi^{n}-I_h\Phi^{n}),\nabla\eta_h^{n}\right)\nonumber\\
        &\qquad + \left((2\sigma(u^{n-1})-\sigma(u^{n-2}))\nabla(\eta_\tau^{n}-I_h\eta_\tau^{n}),\nabla \eta_h^{n}\right)\nonumber\\
        &\qquad+\left((2\sigma(u^{n-1})-\sigma(u^{n-2}))\nabla(\phi^{n}-I_h\phi^{n}),\nabla\eta_h^{n}\right)\nonumber\\
       &\eqqcolon \sum_{\ell=1}^8 A_\ell,\nonumber
    \end{alignat}
where $\eta_\tau^n=\Phi^n-\phi^n$ denotes the temporal error known from the proof of Theorem \ref{thm:TemError}. By \eqref{ineq:sigma}, it is easy to see that the left-hand side satisfies  
\[
    k_*\|\nabla\eta_h^{n}\|_0^2
    \leq\big((2\sigma(U^{n-1})-\sigma(U^{n-2}))\nabla\eta_h^{n},\nabla\eta_{h}^{n}\big),
\]
where $k_*$ is as defined in~\eqref{eq:k*}.
By the inverse inequality  \eqref{ineq:invese:pq}, we obtain 
\begin{align*}
    A_{1}&\leq   2 L_{\sigma,2} \|U^{n-1}-U_h^{n-1}\|_0\,\|\nabla\eta_h^n\|_{\infty}\,\|\nabla \eta_h^{n}\|_0\nonumber\\
    &\leq Ch^{-1}\left(\|U^{n-1}-I_hU^{n-1}\|_0+\|e_h^{n-1}\|_0\right)\|\nabla \eta_h^{n}\|_0^2\\
    &\leq  Ch^{-1}\left(h^2\|U^{n-1}\|_2+\|\nabla e_h^{n-1}\|_0\right)\|\nabla \eta_h^{n}\|_0^2\\
    &\leq \tfrac14\, \|\nabla \eta_h^{n}\|_0^2,
\end{align*}
where we use the bound $\|U^{n-1}\|_2\leq C$ from \eqref{ineq:thm:error2} and require $C(h+\tau^2)\leq \frac14$, which is satisfied for sufficiently small $\tau$ and $h$. Similarly, using \eqref{ineq:Ih} and \eqref{ineq:sta:Ih}, we get
\begin{align*}
    A_{2}
    &\leq 2 L_{\sigma,2} \|U^{n-1}-U_h^{n-1}\|_0\|\Phi^n\|_{1,\infty}\|\nabla \eta_h^{n}\|_0\\
    &\leq C\left(h^2\|U^{n-1}\|_2+\| e_h^{n-1}\|_0\right)\|\Phi^n\|_{2,4}\|\nabla \eta_h^{n}\|_0\\
    &\leq C\, (h^2+h\tau^2)\|\nabla \eta_h^{n}\|_0,
\end{align*}
where we have used the embedding $W^{2,4}(\Omega)\hookrightarrow W^{1,\infty}(\Omega)$
(see~\cite{nirenberg1966} for further details) and $\|\Phi^n\|_{2,4}\leq C$ from \eqref{ineq:thm:error2}. 
For $A_3$ and $A_4$, we obtain the same estimates, respectively. 
By \eqref{ineq:Ih}, \eqref{ineq:thm:error1}, \eqref{ineq:thm:error2}, and the embedding $H^1(\Omega)\hookrightarrow W^{0,4}(\Omega)$,
we have
 \begin{align*}
    A_5&\leq 2 L_{\sigma,4} \|U^{n-1}-u^{n-1}\|_{0,4}\|\nabla(\Phi^n-I_h\Phi^n)\|_{0,4}\|\nabla \eta_h^{n}\|_0\\
    &\leq Ch\, \|U^{n-1}-u^{n-1}\|_1\|\Phi^{n}\|_{2,4}\|\nabla \eta_h^{n}\|_0\\
    &\leq Ch\tau^2\|\nabla \eta_h^{n}\|_0.
\end{align*}
For $A_6$, we obtain the same estimate. 
Again, by using \eqref{ineq:Ih}, \eqref{ineq:thm:error1}, and \eqref{ineq:thm:error1}, we get
\begin{align*}
    A_7 
    &\leq 3 k_2\, \|\nabla(\eta_\tau^{n}-I_h\eta_\tau^{n})\|_0\|\nabla \eta_h^{n}\|_0
    \leq Ch\, \|\eta_\tau^{n}\|_2\|\nabla \eta_h^{n}\|_0
    \leq Ch\tau^2\|\nabla \eta_h^{n}\|_0.
\end{align*}
By \eqref{ineq:lem:Ih} from  Lemma \ref{lem:sup:Ih} and the regularity assumptions \eqref{eq:regularity}, $A_8$ is bounded by
\begin{align*}
    A_8
    \leq 3\,C k_2h^2\|\phi^{n}\|_3\,\|\nabla\eta_h^{n}\|_0
   \leq Ch^2\|\nabla \eta_h^{n}\|_0.
\end{align*}
Combining the obtained estimates  for $A_1,\ldots,A_8$, we have
\begin{equation}\label{ineq:est:eta_h}
    \|\nabla \eta_h^{n}\|_0
    \leq C\, (h^2+h\tau^2).
\end{equation}
From \eqref{eq:BDF:Galerkin:a} and \eqref{eq:timeDis:a} considered with $\xi_u =  D_\tau e_h^{n}$, the spatial error function $e_h^n$ satisfies
    \begin{alignat}{2}
        (D_\tau e_h^{n},D_\tau e_h^{n})+(\nabla e_h^{n},\nabla D_\tau e_h^{n})& =  \left(D_\tau(U^{n}-I_hU^{n}),D_\tau e_h^{n}\right)\nonumber\\
        &\qquad + 2\left((\sigma(U_h^{n-1}) - \sigma(U^{n-1}))|\nabla \Phi^{n}|^2,D_\tau e_h^{n}\right)\nonumber\\
        &\qquad - \left((\sigma(U_h^{n-2}) - \sigma(U^{n-2}))|\nabla \Phi^{n}|^2,D_\tau e_h^{n}\right)\nonumber\\
        &\qquad + \left((2\sigma(U_h^{n-1}) - \sigma(U_h^{n-2}))(|\nabla\Phi_h^{n}|^2-|\nabla\Phi^{n}|^2),D_\tau e_h^{n}\right)\nonumber\\
        &\eqqcolon B_1 + B_2 + B_3 + B_4. \label{eq:errors:spa:e}
    \end{alignat}
By \eqref{ineq:Ih} and Young’s inequality, the first two terms can be bounded by 
\[
    B_1
    \leq h^2\|D_\tau U^{n}\|_2\,\|D_\tau e_h^{n}\|_0
    \leq Ch^4\|D_\tau U^{n}\|_2^2+\tfrac{1}{12}\,\|D_\tau e_h^{n}\|_0^2
\]
and 
\begin{align*}
    B_2&\leq  \|\sigma(U_h^{n-1})-\sigma(U^{n-1})\|_0\,\|\nabla \Phi^{n}\|_{\infty}^2\,\|D_\tau   e_h^{n}\|_0\\
    &\leq  C L_{\sigma,2}\left(h^2\|U^{n-1}\|_2+\|\nabla e_h^{n-1}\|_0\right)\| D_\tau  e_h^{n}\|_0\\
    &\leq C \big(h^4+\|\nabla e_h^{n-1}\|_0^2 \big) + \tfrac{1}{12}\, \|D_\tau   e_h^{n}\|_0^2,
\end{align*}
where we have used again the embedding $W^{2,4}(\Omega)\hookrightarrow W^{1,\infty}(\Omega)$ and $\|U^{n-1}\|_2\leq C$ from \eqref{ineq:thm:error2}. Similarly, it holds
\begin{align*}
    B_3
    \leq C \big(h^4+\|\nabla e_h^{n-2}\|_0^2\big) + \tfrac{1}{12}\, \|D_\tau e_h^{n}\|_0^2.
\end{align*}
Note that $B_4$ can be rewritten as 
\begin{align*}
    B_4&=  4\left((\sigma(U_h^{n-1})-\sigma(U^{n-1}))\nabla \Phi^{n}\cdot \nabla(\Phi_h^{n}-\Phi^{n}),  D_\tau  e_h^{n}\right) \\
    &\qquad -2\left((\sigma(U_h^{n-2})-\sigma(U^{n-2}))\nabla \Phi^{n}\cdot \nabla(\Phi_h^{n}-\Phi^{n}),  D_\tau  e_h^{n}\right) \\
    &\qquad+ \left((2\sigma(U_h^{n-1})-\sigma(U_h^{n-2}))|\nabla(\Phi_h^{n}-\Phi^{n})|^2, D_\tau e_h^{n}\right) \\
    &\qquad+ 2\left((2\sigma(U^{n-1})-\sigma(U^{n-2}))\nabla \Phi^{n}\cdot\nabla (\Phi_h^{n}-I_h\Phi^{n}),   D_\tau e_h^{n}\right)\\
    &\qquad  + 2\left((2\sigma(U^{n-1})-\sigma(U^{n-2}))\nabla \Phi^{n}\cdot\nabla (I_h\Phi^{n}-\Phi^{n}),   D_\tau e_h^{n}\right)\\
    &\eqqcolon  B_{41} + B_{42} +B_{43} +B_{44} + B_{45}. 
\end{align*}
By \eqref{ineq:Ih}, the inverse inequality  \eqref{ineq:invese:pq} with $m=0$, $p=4$, $q=2$, \eqref{ineq:thm:error2}, and \eqref{ineq:est:eta_h} we have 
\begin{align*}
    B_{41}&\leq  4 L_{\sigma,4} \|U_h^{n-1}-U^{n-1}\|_{0,4}\,\|\nabla \Phi^{n}\|_{\infty}\,\|\nabla(\Phi_h^{n}-\Phi^{n})\|_{0,4}\,\| D_\tau   e_h^{n}\|_0\\
    &\leq  C\big(\|e_h^{n-1}\|_{0,4}+h\|U^{n-1}\|_2\big)\big(h \|\Phi^{n}\|_{2,4}+\|\nabla \eta _h^{n}\|_{0,4}\big)\|D_\tau   e_h^{n}\|_0\\
    &\leq  C\big(\|\nabla e_h^{n-1}\|_0+h\big)\big(h+h^{-\frac12}(h^2+h\tau^2)\big)\|D_\tau    e_h^{n}\|_0.
    \end{align*}
    Hence, Young's inequality yields
    \begin{align*}
    B_{41}&\leq   C\big(\|\nabla e_h^{n-1}\|_0+h^2+h\tau^2\big)\|  D_\tau e_h^{n}\|_0
    \leq \, C\left(\|\nabla e_h^{n-1}\|_0^2 + (h^2+h\tau^2)^2\right)+\tfrac{1}{12}\, \|D_\tau    e_h^{n}\|_0^2.
    \end{align*}
  Similarly, one shows
\begin{align*}
    B_{42}
    \leq C\left((h^2+h\tau^2)^2+\|\nabla e_h^{n-2}\|_0^2\right)+\tfrac{1}{12}\, \|D_\tau    e_h^{n}\|_0^2.
\end{align*}
Using the inverse inequality \eqref{ineq:invese:pq} and \eqref{ineq:est:eta_h}, it holds that
\begin{align*}
    B_{43}
    &\leq 3k_2\|\nabla(\Phi_h^{n}-\Phi^{n})\|_{0,4}^2\|  D_\tau e_h^{n}\|_0 \\
    &\leq C\, (h^2+h\tau^2) \| D_\tau e_h^{n}\|_0 \\
    &\leq C\, (h^2+h\tau^2)^2+\tfrac{1}{12}\, \| D_\tau  e_h^{n}\|_0^2,
\end{align*}
where we have used $\|\nabla(\Phi_h^{n}-\Phi^{n})\|_{0,4}\leq C(h+h^{1/2}\tau^2)$ as in the estimate of $B_{41}$ before. 
By the embedding $W^{2,4}(\Omega)\hookrightarrow W^{1,\infty}(\Omega)$, \eqref{ineq:thm:error2},  \eqref{ineq:est:eta_h}, and Young's inequality,  we obtain
\begin{align*}
    B_{44}&\leq  6k_2\|\nabla \Phi^n\|_{\infty}\|\nabla(\Phi_h^n-I_h\Phi^n)\|_{0}\|D_\tau e_h^n\|_{0}\\
    &\leq  C\,\|\Phi^n\|_{2,4}\|\nabla\eta_h^n\|_0\|D_\tau e_h^n\|_0\\
    &\leq  C(h^2+h\tau^2)^2+\tfrac{1}{12}\,\|D_\tau e_h^n\|_0^2.
\end{align*}
To estimate $B_{45}$, we rewrite it as 
\begin{align*}
    B_{45} &=  4\big((\sigma(U^{n-1})-\sigma(u^{n-1}))\nabla \Phi^n\cdot \nabla(I_h\Phi^n-\Phi^n),D_\tau e_h^n\big)\\
    &\qquad -2\big((\sigma(U^{n-2})-\sigma(u^{n-2}))\nabla \Phi^n\cdot \nabla(I_h\Phi^n-\Phi^n),D_\tau e_h^n\big)\\
    &\qquad +2\big((2\sigma(u^{n-1})-\sigma(u^{n-2}))\nabla \eta_\tau^n\cdot \nabla(I_h\Phi^n-\Phi^n),D_\tau e_h^n\big)\\
    &\qquad +2\big((2\sigma(u^{n-1})-\sigma(u^{n-2}))\nabla \phi^n\cdot \nabla(I_h\eta_\tau^n-\eta_\tau^n),D_\tau e_h^n\big)\\
    &\qquad -2\big((\sigma(u^n)-2\sigma(u^{n-1})+\sigma(u^{n-2}))\nabla \phi^n\cdot\nabla (I_h\phi^n-\phi^n),D_\tau e_h^n\big)\\
    &\qquad +2\big(\sigma(u^n)\nabla \phi^n\cdot\nabla (I_h\phi^n-\phi^n),D_\tau e_h^n\big)\\
    &\eqqcolon \sum_{\ell =1}^6 D_{\ell}.
\end{align*}
Using \eqref{ineq:Ih}, \eqref{ineq:thm:error1}, \eqref{ineq:thm:error2}, Sobolev embeddings, and Young's inequality, the first term is bounded by 
\begin{align*}
    D_{1}
    &\leq 4 L_{\sigma,4}\|U^{n-1}-u^{n-1}\|_{0,4}\|\nabla\Phi^n\|_{\infty}\|\nabla(I_h\Phi^n-\Phi^n)\|_{0,4}\|D_\tau e_h^n\|_{0}\\
    &\leq Ch\, \|U^{n-1}-u^{n-1}\|_1\|\Phi^n\|_{2,4}\|D_\tau e_h^n\|_{0}\\
    &\leq Ch\tau^2\|D_\tau e_h^n\|_{0}\\
    &\leq C (h\tau^2)^2+\tfrac{1}{12}\, \|D_\tau e_h^n\|_{0}^2.
\end{align*}
A similar estimate holds for $D_{2}$. Using the same estimates as before, we deduce 
\begin{align*}
    D_{3}&\leq 6k_2\|\nabla \eta_\tau^n\|_{0,4}\|\nabla(I_h\Phi^n-\Phi^n)\|_{0,4}\|D_\tau e_h^n\|_{0}\\
    &\leq  Ch\|\eta_\tau^n\|_{2}\|\Phi^n\|_{2,4}\|D_\tau e_h^n\|_{0}\\
    &\leq  C(h\tau^2)^2+\tfrac{1}{12}\,\|D_\tau e_h^{n}\|_0^2,
\end{align*}
where we have used the embedding $H^2(\Omega)\hookrightarrow W^{1,4}(\Omega)$. Moreover, we have
\begin{align*}
    D_{4}
    \leq 6k_2\|\nabla \phi^n\|_{\infty}\|\nabla(I_h\eta_\tau^n-\eta_\tau^n)\|_{0}\|D_\tau e_h^n\|_0
    \leq Ch\, \|\eta_\tau^n\|_2\|D_\tau e_h^n\|_0
    \leq C(h\tau^2)^2+\tfrac{1}{12} \|D_\tau e_h^n\|_{0}^2.
\end{align*}
By \eqref{ineq:Ih} and Young's inequality, one further obtains 
\begin{align*}
    D_5
    &\leq 2\tau^2\|{\sigma}''(u^{n-1})\|_{0}\|\nabla \phi^n\|_{\infty}\|\nabla (I_h\phi^n-\phi^n)\|_{0,4}\|D_\tau e_h^n\|_{0}\\
    &\leq Ch\tau^2\|\phi^n\|_{2,4}\|D_\tau e_h^n\|_{0}\\
    &\leq C(h\tau^2)^2 + \tfrac{1}{12}\,\|D_\tau e_h^n\|_{0}^2.
\end{align*}
For the last term $D_{6}$, we use the identity $2(a^n,D_\tau b^n)=3(a^n,\widetilde{D}_\tau b^n)-(a^n,\widetilde{D}_\tau b^{n-1})$, where $\widetilde{D}_\tau f^n \coloneqq\tau^{-1}({f^n-f^{n-1}})$ denotes the first-order difference quotient. Therefore,   we have 
\begin{align*}
    D_{6} &= 2\big(\sigma(u^n)\nabla \phi^n\cdot\nabla (I_h\phi^n-\phi^n),D_\tau e_h^n\big)\\
          &= 3 \big(\sigma(u^n)\nabla \phi^n\cdot\nabla (I_h\phi^n-\phi^n),\widetilde{D}_\tau e_h^n\big)
          -\big(\sigma(u^n)\nabla \phi^n\cdot\nabla (I_h\phi^n-\phi^n),\widetilde{D}_\tau e_h^{n-1}\big)\\
          &\eqqcolon D_{61} +D_{62}.
\end{align*}
Now, using the fact that $(a^n,\widetilde{D}_\tau b^n)=\widetilde{D}_\tau(a^n, b^n)-(\widetilde{D}_\tau a^n,b^{n-1})$ and an estimation as in~\cite[p.~15]{yang2023}, $D_{61}$ can be bounded by
\begin{align*}
    D_{61} 
    \leq  3 \widetilde{D}_\tau\big(\sigma(u^n)\nabla \phi^n\cdot\nabla (I_h\phi^n-\phi^n), e_h^n\big) + Ch^2\|\nabla e_h^{n-1}\|_{0}.
\end{align*}    
Analogously, one obtains the estimate
\[
    D_{62}
    \leq -\widetilde{D}_\tau\big(\sigma(u^n)\nabla \phi^n\cdot\nabla (I_h\phi^n-\phi^n), e_h^{n-1}\big) + Ch^2\|\nabla e_h^{n-2}\|_{0}.
\]
With these estimates and Young's inequality, it holds that 
\begin{align*}
     B_{45}&\leq   C\big((h^2+h\tau^2)^2+\|\nabla e_h^{n}\|_{0}^2 +\|\nabla e_h^{n-1}\|_{0}^2 + \|\nabla e_h^{n-2}\|_{0}^2\big)+\tfrac{5}{12}\,\|D_\tau e_h^n\|_0^2 \\
     &\qquad +  3 \widetilde{D}_\tau\big(\sigma(u^n)\nabla \phi^n\cdot\nabla (I_h\phi^n-\phi^n), e_h^n\big) -\widetilde{D}_\tau\big(\sigma(u^n)\nabla \phi^n\cdot\nabla (I_h\phi^n-\phi^n), e_h^{n-1}\big).
\end{align*}
From the estimates of $B_{41},\ldots,B_{45}$, it follows that 
\begin{align*}
    B_4
    &\leq C\left((h^2+h\tau^2)^2+\|\nabla e_h^{n}\|_0^2+\|\nabla e_h^{n-1}\|_0^2+\|\nabla e_h^{n-2}\|_0^2\right)+\tfrac{9}{12}\,\|D_\tau e_h^{n}\|_0^2\\
    &\qquad+ 3 \widetilde{D}_\tau\big(\sigma(u^n)\nabla \phi^n\cdot\nabla (I_h\phi^n-\phi^n), e_h^n\big) -\widetilde{D}_\tau\big(\sigma(u^n)\nabla \phi^n\cdot\nabla (I_h\phi^n-\phi^n), e_h^{n-1}\big).
\end{align*}
Substituting the estimates of $B_1,\ldots,B_4$ into \eqref{eq:errors:spa:e} leads to 
\begin{align*}
    \|D_\tau  e_h^{n}\|_0^2\,+ \left( \nabla  e_h^{n},D_\tau\nabla e_h^{n}\right)
    &\leq Ch^4\|D_\tau U^{n}\|_2^2 +  C(h^2+h\tau^2)^2\,+\|D_\tau e_h^{n}\|_0^2\\
    &\qquad +C\left(\|\nabla e_h^{n}\|_0^2+\|\nabla e_h^{n-1}\|_0^2+\|\nabla e_h^{n-2}\|_0^2\right)\\
    &\qquad +  3 \widetilde{D}_\tau\big(\sigma(u^n)\nabla \phi^n\cdot\nabla (I_h\phi^n-\phi^n), e_h^n\big)\\
    &\qquad -\widetilde{D}_\tau\big(\sigma(u^n)\nabla \phi^n\cdot\nabla (I_h\phi^n-\phi^n), e_h^{n-1}\big).
\end{align*}
By summing up, using the telescope formula for $D_\tau$ from Lemma \ref{lem:tel:BDF}, and  \eqref{ineq:thm:error2}, we obtain
\begin{align}
    \|\nabla e_h^{n}\|_0^2&\leq  C(h^2+h\tau^2)^2+C\tau \sum_{\ell=1}^{n}\|\nabla e_h^{\ell}\|_0^2\nonumber\\
    &\qquad + 3 \big(\sigma(u^n)\nabla \phi^n\cdot\nabla (I_h\phi^n-\phi^n), e_h^n\big) -\big(\sigma(u^n)\nabla \phi^n\cdot\nabla (I_h\phi^n-\phi^n), e_h^{n-1}\big).\label{ineq:last}
\end{align}
By \eqref{ineq:lem:Ih2} from Lemma \ref{lem:sup:Ih}, the regularity assumptions \eqref{eq:regularity}, and the fact that $e_h^n\in \Vo$, one shows 
\begin{align*}
    3\big(\sigma(u^n)\nabla \phi^n\cdot\nabla (I_h\phi^n-\phi^n), e_h^n\big) 
    &\leq Ch^2\|\phi^n\|_3\|e_h^n\|_1\leq Ch^4+\tfrac12\,\|\nabla e_h^{n}\|_0^2.
\end{align*}
Using additionally \eqref{ineq:priEst:spa}, we get
\begin{align*}
    -\big(\sigma(u^n)\nabla \phi^n\cdot\nabla (I_h\phi^n-\phi^n), e_h^{n-1}\big) 
    &\leq Ch^2\|\phi^n\|_3\|e_h^{n-1}\|_1 \\
    &\leq Ch^4+\tfrac12\, \|\nabla e_h^{n-1}\|_0^2
    \leq C\, (h^2+h\tau^2)^2.
\end{align*}
Substituting these estimates into \eqref{ineq:last}, we get 
 \begin{align*}
    \|\nabla e_h^{n}\|_0^2
    &\leq C\, (h^2+h\tau^2)^2+C\tau \sum_{\ell=1}^{n}\|\nabla e_h^{\ell}\|_0^2 + \tfrac{1}{2}\, \|\nabla e_h^n\|_{0}^2.
\end{align*}
Applying Gronwall’s inequality from Lemma \ref{lem:Gronwall}, we obtain
\begin{align*}
    \|\nabla e_h^{n}\|_0^2  &\leq  C\exp\left(\tau \sum_{\ell=0}^{n}\frac{C}{1-C\tau}\right)(h^2+h\tau^2)^2 \leq \, C\exp\left(2CT\right)(h^2+h\tau^2)^2.
\end{align*}
Hence, estimate \eqref{ineq:priEst:spa} holds for $n\geq 0$. Combining \eqref{ineq:priEst:spa} and \eqref{ineq:est:eta_h} together with Poincaré's inequality, the desired result follows. 
\end{proof}
%
%
\section{Convergence Results}\label{Sec:main}
Based on Theorems \ref{thm:TemError} and \ref{thm:err:spatial}, we obtain the following error estimates. 
\begin{theorem}(Convergence of the implicit--explicit BDF--Galerkin scheme)
\label{thm:main}
Suppose that system~\eqref{eq:thermistor} has a unique solution $(u,\phi)$ satisfying \eqref{eq:regularity} and that $\tau$ and $h$ are sufficiently small. Then the discrete system \eqref{eq:BDF:Galerkin} with the starting values $U_h^1$, $\Phi_h^1$ from \eqref{eq:int:Uh11_Phi11} admits a unique solution $(U_h^n,\Phi_h^n)$, which satisfies the error bounds 
\begin{subequations}
  \label{ineq:mainT}
    \begin{alignat}{2}
        \|U_h^n-u^n\|_0 \, + \, \|\Phi_h^n-\phi^n\|_0
        &\leq \, C \big( h^2+\tau^2 \big),\label{ineq:mainT:a}\\
        \|U_h^n-u^n\|_1 \, + \, \|\Phi_h^n-\phi^n\|_1
        &\leq \, C \big( h+\tau^2 \big),\label{ineq:mainT:b}\\
        \|U_h^n-I_hu^n\|_1\, + \, \|\Phi_h^n-I_h\phi^n\|_{1}
        &\leq \, C \big( h^2+\tau^2 \big).\label{ineq:mainT:c}
        \end{alignat}
\end{subequations}
\end{theorem}
\begin{proof}
By \eqref{ineq:sta:Ih}, \eqref{ineq:thm:error1}, and \eqref{ineq:err:full}, we have
\begin{align*}
    \|U_h^n-I_hu^n\|_1
    &\leq \|U_h^n-I_hU^n\|_1+\|I_hU^n-I_hu^n\|_1\nonumber\\
    &\leq \|U_h^n-I_hU^n\|_1+C\,\|U^n-u^n\|_1
    \leq C\, (h^2+\tau^2).\label{ineq:thm1:a}
\end{align*}
On the other hand, by triangle inequality, the previous estimate implies
\begin{equation*}
    \|U_h^n-u^n\|_1
    \leq \|U_h^n-I_hu^n\|_1+\|I_hu^n-u^n\|_1
    \leq C\, (h+\tau^2)
\end{equation*}
It follows
\begin{alignat*}{2}
    \|U_h^n-u^n\|_0
    &\leq  \|U_h^n-I_hu^n\|_0+\|I_hu^n-u^n\|_0\\
    &\leq C\,\|U_h^n-I_hu^n\|_1+Ch^2\|u^n\|_2
    \leq C\, (h^2+\tau^2).
\end{alignat*}
Similarly, we obtain the claimed results for $\phi$.
\end{proof} 
Estimate \eqref{ineq:mainT:c} demonstrates that the fully discrete solution is superclose to the interpolated exact solution in the $H^1$-norm, i.e., the discrete solution approximates the interpolated exact solution more accurately than the solution itself. This supercloseness property can be used to obtain global superconvergence through suitable post-processing.  

Following the approach in~\cite{yang2023}, we apply an interpolation based post-processing technique to enhance the accuracy of the fully discrete solution. To this end, we consider a local post-processing operator~$I_{2h}\colon C(\widetilde{K}) \to \mathcal{P}(\widetilde{K})$ as introduced in \cite{lin2007} on a {macroelement} $\widetilde{K}$, which consists of four standard elements, cf.~Figure~\ref{fig:macroelement}. This corresponds to a coarser mesh with mesh size~$2h$.
For rectangular macroelements, we set~$\mathcal{P}(\widetilde{K}) = Q_2^2(\widetilde{K})$, the space of bi-quadratic polynomials. The operator~$I_{2h}$ preserves nodal values at the nine vertices of each macroelement and maps into the space of polynomials of degree at most two in each variable. 
For triangular macroelements, we set $\mathcal{P}(\widetilde{K}) = {P}_2(\widetilde{K})$ with the space of polynomials of total degree at most two. In this case, the operator preserves nodal values at the six vertices of each macroelement. 
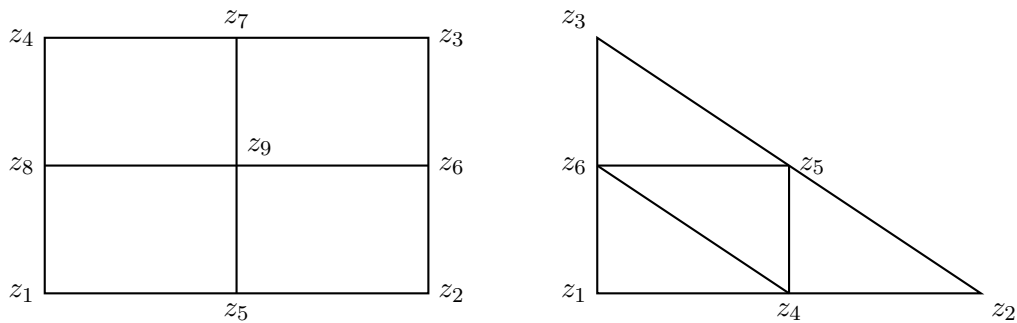
\begin{figure}
    \centering
    \begin{tikzpicture}[scale=0.85]
        \draw[thick] (0,0) rectangle (6,4);
        \draw[thick] (3,0) -- (3,4);
        \draw[thick] (0,2) -- (6,2);
        \node[left] at (0,0) {$ z_1 $};
        \node[right] at (6,0) {$ z_2 $};
        \node[right] at (6,4) {$ z_3 $};
        \node[left] at (0,4) {$ z_4 $};
        \node[anchor=south west] at (3,2) {$ z_9 $};
        \node[below] at (3,0) {$ z_5 $};
        \node[right] at (6,2) {$ z_6 $};
        \node[above] at (3,4) {$ z_7 $};
        \node[left] at (0,2) {$ z_8 $};
    \end{tikzpicture}
    \hspace{2em}
    \begin{tikzpicture}[scale=0.85]
        \coordinate (A) at (0,0);
       \coordinate (B) at (6,0);
        \coordinate (C) at (0,4);
        \coordinate (D) at (3,0); 
       \coordinate (E) at (3,2); 
        \coordinate (F) at (0,2); 
        \draw[thick] (A) -- (B) -- (C) -- cycle;
        \draw[thick] (D) -- (F);
        \draw[thick] (D) -- (E);
        \draw[thick] (E) -- (F);
        \node[left] at (A) {$z_1$};
        \node[below right] at (B) {$z_2$};
        \node[above left] at (C) {$z_3$};
        \node[below] at (D) {$z_4$};
        \node[right] at (E) {$z_5$};
        \node[left] at (F) {$z_6$};
    \end{tikzpicture}
    \caption{Macroelement $\widetilde{K}$ for a quadrilateral and a triangle.}
    \label{fig:macroelement}
\end{figure}

This post-processing step improves the approximation quality by using the supercloseness property of the finite element solution. While the original approximation is only first-order accurate in the $H^1$-norm, the post-processed solution on the coarser mesh achieves a higher order of convergence.  
We derive a global superconvergence estimate under suitable regularity assumptions, based on the approximation and stability properties of the interpolation operator~$I_{2h}$, cf.~\cite{lin2007}. Specifically, it holds 
\begin{subequations}
\label{eq:I_2h}
    \begin{align}
        I_{2h} I_h \omega\, 
        &=\, I_{2h} \omega,\label{eq:I_2h:a}\\
        \| \omega - I_{2h} \omega \|_1 \,
        &\leq\,  C h^2\, \| \omega \|_3,\label{eq:I_2h:b}\\
        \| I_{2h} v_h \|_1 \,
        &\leq\,  C\, \| v_h \|_1 \label{eq:I_2h:c}
    \end{align}
\end{subequations}
for all $\omega \in H^3(\Omega)$ and $v_h \in \Vo$. 
\begin{theorem}(Superconvergence)
\label{thm:global}
Suppose that \eqref{eq:regularity} is satisfied. Then, we have 
\begin{equation*}
    \|I_{2h}U_h^n-u^n\|_1\, +\, \|I_{2h}\Phi_h^n-\phi^n\|_1
    \leq C\, (h^2+\tau^2).
\end{equation*}
\end{theorem}
    \begin{proof}
        According to \eqref{ineq:mainT:c} and the properties of the operator $I_{2h}$ given in~\eqref{eq:I_2h}, one can verify that 
        \begin{alignat*}{2}
            \|I_{2h}U_h^n-u^n\|_1&\leq  \|I_{2h}U_h^n-I_{2h}I_hu^n\|_1+\|I_{2h}I_hu^n-u^n\|_1\\
            &\leq \|I_{2h}(U_h^n-I_hu^n)\|_1+\|I_{2h}u^n-u^n\|_1\\
            &\leq C\,\|I_hu^n-U_h^n\|_1 + Ch^2\|u^n\|_3\\
            &\leq C\, (h^2+\tau^2),
        \end{alignat*}
which establishes the desired result for $u$. The estimate for $\phi$ can be derived similarly.
\end{proof}
%
%
\section{Numerical results}\label{Sec.Numeric}
In this section, we present numerical experiments that validate the theoretical analysis and illustrate the performance of the proposed method. Moreover, we study a number of examples beyond the theoretical basis. In all computations, the final time is set to $T=1.0$.
%
\subsection{Convergence and stability}  
In this first example, we verify the spatial and temporal convergence properties of the proposed method using a manufactured solution. To this end, artificial source terms are introduced so that the exact solution is known. As in~\cite{li2014,gao2016}, we consider the problem
\begin{subequations}\label{eq:thermistor:f}
    \begin{align}
    \dot{u}-\Delta u 
    &= \sigma(u)|\nabla \phi|^2 + f_1,\\
    -\nabla \cdot(\sigma(u)\nabla \phi) 
    &= f_2
\end{align}
\end{subequations}
in $\Omega = (0,1)^2$ with electric conductivity
\[
    \sigma(u) = \frac{1}{1+u^2}+1.
\]
Furthermore, the functions $f_1$ and $f_2$, along with the initial and boundary conditions are determined corresponding to the manufactured solution
\[
    u(x,y,t) 
    = e^{-2t}\sin(\pi x)\sin(\pi y),\qquad 
    \phi(x,y,t) 
    = 1 + \sin(x+y+t).
\]
%
We partition the domain $\Omega$ into squares with grid resolutions of $M\times M$ with $M = 32, 64, 128, 256$. The mesh size is given by $h = \sqrt{2}/M$. To verify the error estimates presented in Theorems~\ref{thm:main} and~\ref{thm:global}, we first set the time step size to $\tau = h$. 
The convergence, supercloseness, and superconvergence results for $u$ and $\phi$ at time $t=T$ are reported in Tables~\ref{tab:u_convergence} and~\ref{tab:phi_convergence}, respectively. From Table~\ref{tab:u_convergence}, we observe that $\|U_h^n - u^n \|_1$ converges at a rate of $\calO(h)$, while $\| U_h^n -u^n \|_0$, $\| U_h^n - I_h u^n \|_1$, and $\| I_{2h} U_h^n - u^n\|_1$ exhibit convergence rates of order two as expected. Similarly, Table~\ref{tab:phi_convergence} shows that $\| \Phi_h^n - \phi^n\|_1$ converges with order one, whereas $\| \Phi_h^n - \phi^n \|_0$, $\|  \Phi_h^n - I_h \phi^n  \|_1$, and $\| I_{2h} \Phi_h^n - \phi^n \|_1$ converge at rate $\calO(h^2)$. Note that all these numerical results align with the theoretical analysis of Section~\ref{Sec:main}. 
\begin{table}
\centering
\caption{Numerical errors for $u$ at final time $t=T$ using the second-order BDF--Galerkin FEM~\eqref{eq:BDF:Galerkin} and $\tau = h$.}
\label{tab:u_convergence}
\begin{tabular}{c|cccc}
\toprule
\text{mesh}  & $32\times 32$ & $64\times 64$ & $128\times 128$ & $256\times 256$\\
\midrule
$\|U_h^n - u^n\|_0$ 
& $8.33\cdot 10^{-5}$ 
& $2.07\cdot 10^{-5}$ 
& $5.23\cdot 10^{-6}$ 
& $1.31\cdot 10^{-6}$ \\
order 
& -- 
& $2.01$ 
& $1.98$ 
& $2.00$ \\
\addlinespace
$\|U_h^n - u^n\|_1$ 
& $1.42\cdot 10^{-2}$ 
& $7.13\cdot 10^{-3}$ 
& $3.65\cdot 10^{-3}$ 
& $1.82\cdot 10^{-3}$ \\
order 
& -- 
& $1.00$ 
& $0.97$ 
& $1.00$ \\
\addlinespace
$\|U_h^n - I_h u^n \|_1$ 
& $3.90\cdot 10^{-4}$ 
& $9.52\cdot 10^{-5}$ 
& $2.39\cdot 10^{-5}$ 
& $5.94\cdot 10^{-6}$ \\
order 
& -- 
& $2.03$ 
& $1.99$ 
& $2.01$ \\
\addlinespace
$\|I_{2h} U_h^n - u^n \|_1$ 
& $2.64\cdot 10^{-3}$ 
& $6.13\cdot 10^{-4}$ 
& $1.50\cdot 10^{-4}$ 
& $3.65\cdot 10^{-5}$ \\
order 
& -- 
& $2.11$ 
& $2.03$ 
& $2.03$ \\
\bottomrule
\end{tabular}
\end{table}

\begin{table}
\centering
\caption{Numerical errors for $\phi$ at final time $t=T$ using the second-order BDF--Galerkin FEM~\eqref{eq:BDF:Galerkin} and $\tau = h$.}
\label{tab:phi_convergence}
\begin{tabular}{c|cccc}
\toprule
\text{mesh}  & $32\times 32$ & $64\times 64$ & $128\times 128$ & $256\times 256$\\
\midrule
$\|\Phi_h^n - \phi^n\|_0$ 
& $6.75\cdot 10^{-6}$ 
& $1.56\cdot 10^{-6}$ 
& $3.86\cdot 10^{-7}$ 
& $9.46\cdot 10^{-8}$ \\
order 
& -- 
& $2.11$ 
& $2.02$ 
& $2.03$ \\
\addlinespace
$\|\Phi_h^n - \phi^n\|_1$ 
& $1.87\cdot 10^{-2}$ 
& $9.37\cdot 10^{-3}$ 
& $4.71\cdot 10^{-3}$ 
& $2.36\cdot 10^{-3}$ \\
order 
& -- 
& $1.00$ 
& $0.99$ 
& $1.00$ \\
\addlinespace
$\|\Phi_h^n - I_h \phi^n \|_1$ 
& $4.77\cdot 10^{-5}$ 
& $1.10\cdot 10^{-5}$ 
& $2.69\cdot 10^{-6}$ 
& $6.58\cdot 10^{-7}$ \\
order 
& -- 
& $2.12$ 
& $2.03$ 
& $2.03$ \\
\addlinespace
$\|I_{2h} \Phi_h^n - \phi^n \|_1$ 
& $6.38\cdot 10^{-4}$ 
& $1.42\cdot 10^{-4}$ 
& $3.29\cdot 10^{-5}$ 
& $7.93\cdot 10^{-6}$ \\
order 
& -- 
& $2.17$ 
& $2.11$ 
& $2.05$ \\
\bottomrule
\end{tabular}
\end{table}

To verify the unconditional convergence and accuracy of the proposed method, we conduct two complementary simulations. 
First, we fix the time step size $\tau$ 
and only refine the spatial mesh. The resulting errors for $u$ and $\phi$ (for different values of $\tau$) are displayed in Figure~\ref{Fig:test1:errors:M}, demonstrating that the numerical errors stabilize as $h/\tau\rightarrow0$ and that the proposed method does not require a time step restriction. 
Second, to show the temporal order of convergence, we fix the mesh size $h$, and reduce $\tau$. The $L^2$-errors presented in Figure~\ref{Fig:test1:errors:tau} show second-order convergence, consistent with the theoretical properties of the proposed BDF--Galerkin method. Together, these results verify that the scheme achieves unconditional and optimal convergence in space and time. 
\begin{figure}
    \centering
    \begin{tikzpicture}

\begin{axis}[%
width=2.8in,
height=1.7in,
scale only axis,
xmode=log,
xlabel={$M$},
ylabel={\(\|(U_h^N, \Phi_h^N) - (u^N, \phi^N)\|_0\)},
xmin=7,
xmax=290,
xminorticks=true,
ymode=log,
ymin=1e-6,
ymax=2e-3,
yminorticks=true,
axis background/.style={fill=white},
legend style={at={(1.48,1.0)}, legend cell align=left, align=left, draw=white!15!black},
legend columns=1
]

\addplot [color=mycolor1, mark=square*, line width=1.5pt]
  table[row sep=crcr]{%
    8    0.001200\\
    16   0.000388\\
    32   0.000193\\
    64   0.000146\\
    128  0.000134\\
    256  0.000131\\
};
\addlegendentry{\(\tau=0.1\)}

\addplot [color=mycolor2, mark=*, line width=1.5pt]
  table[row sep=crcr]{%
    8    0.001108\\
    16   0.000293\\
    32   0.000094\\
    64   0.000046\\
    128  0.000034\\
    256  0.000031\\
};
\addlegendentry{\(\tau=0.05\)}

\addplot [color=mycolor3, mark=triangle*, line width=1.5pt, mark size=2.8pt]
  table[row sep=crcr]{%
    8    0.001087\\
    16   0.000272\\
    32   0.000073\\
    64   0.000023\\
    128  0.000011\\
    256  0.000008\\
};
\addlegendentry{\(\tau=0.025\)}

\addplot [color=mycolor4, mark=pentagon*, line width=1.5pt]
  table[row sep=crcr]{%
    8    0.001082\\
    16   0.000267\\
    32   0.000068\\
    64   0.000018\\
    128  0.000006\\
    256  0.000003\\
};
\addlegendentry{\(\tau=0.0125\)}

\addplot[dashed, black, line width=1.0pt] coordinates {(8, 5e-4) (128, 1.953125e-6)};
\addlegendentry{order 2}

\end{axis}

\end{tikzpicture}
    \caption{Convergence history of the $L^2$-errors at time $t=T$ on gradually refined spatial meshes for different but fixed $\tau$.}\label{Fig:test1:errors:M}
\end{figure}
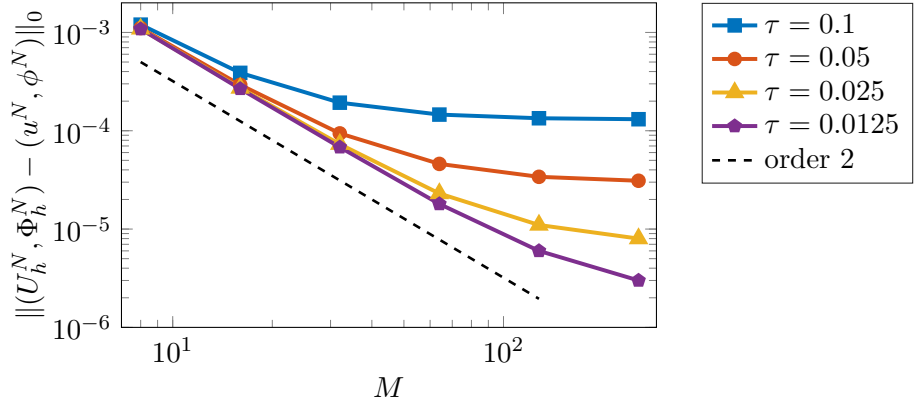
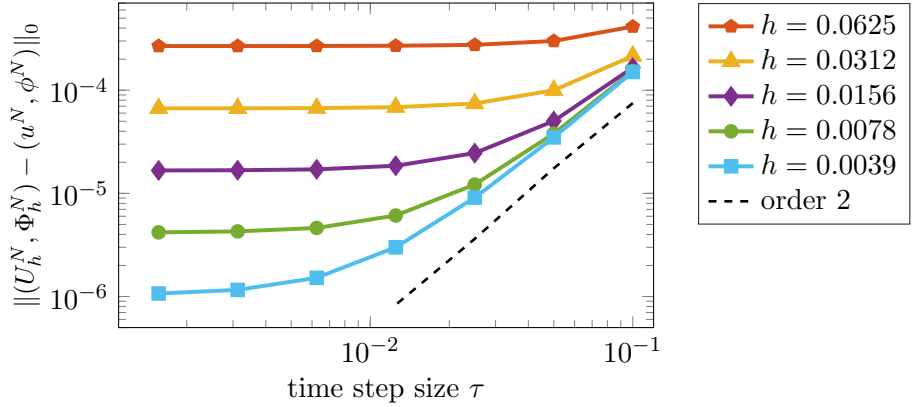
\begin{figure}
\centering
    \begin{tikzpicture}
\begin{axis}[%
width=2.8in,
height=1.7in,
scale only axis,
xmode=log,
xlabel={time step size $\tau$},
ylabel={\small$\|(U_h^N,\Phi_h^N)-(u^N,\phi^N)\|_0$},
xmin=0.0011,
xmax=0.12,
xminorticks=true,
ymode=log,
ymin=0.0000005,
ymax=0.0007,
yminorticks=true,
axis background/.style={fill=white},
legend style={at={(1.48,1.0)}, legend cell align=left, align=left, draw=white!15!black},
legend columns=1
]


\addplot [color=mycolor2, mark=pentagon*, line width=1.5pt]
  table[row sep=crcr]{%
0.1         0.00041381\\
0.05        0.00030004\\
0.025       0.00027519\\
0.0125      0.00026969\\
0.00625     0.00026841\\
0.003125    0.00026810\\
0.0015625   0.00026802\\
};
\addlegendentry{$h=0.0625$} 

\addplot [color=mycolor3, mark=triangle*, line width=1.5pt, mark size=2.8pt]
  table[row sep=crcr]{%
0.1         0.00021481\\
0.05        0.00009993\\
0.025       0.00007435\\
0.0125      0.00006845\\
0.00625     0.00006710\\
0.003125    0.00006679\\
0.0015625   0.00006671\\
};
\addlegendentry{$h= 0.0312$} 

\addplot [color=mycolor4, mark=diamond*, line width=1.5pt, mark size=2.8pt]
  table[row sep=crcr]{%
0.1         0.00016549\\
0.05        0.00005031\\
0.025       0.00002457\\
0.0125      0.00001852\\
0.00625     0.00001709\\
0.003125    0.00001676\\
0.0015625   0.00001668\\
};
\addlegendentry{$h=0.0156$} 

\addplot [color=mycolor5, mark=*, line width=1.5pt]
table[row sep=crcr]{%
0.1         0.00015321\\
0.05        0.00003796\\
0.025       0.00001218\\
0.0125      0.00000610\\
0.00625     0.00000462\\
0.003125    0.00000428\\
0.0015625   0.00000419\\
};
\addlegendentry{$h=0.0078$} 

\addplot [color=mycolor6, mark=square*, line width=1.5pt]
table[row sep=crcr]{%
0.1        0.00015014\\
0.05       0.00003489\\
0.025      0.00000909\\
0.0125     0.00000300\\
0.00625    0.00000152\\
0.003125   0.00000116\\
0.0015625  0.00000107\\
};
\addlegendentry{$h=0.0039$}

\addplot [color=black, dashed, line width=1.0pt]
table[row sep=crcr]{%
0.1       0.000075\\
0.05      0.0000175\\
0.025     0.000003625\\
0.0125    0.00000083125\\
};
\addlegendentry{order $2$} 

\end{axis}
\end{tikzpicture}
    \caption{Convergence history of the $L^2$-errors in $u$ and $\phi$ for fixed spatial meshes and refinements of the time step size~$\tau$.}\label{Fig:test1:errors:tau}
\end{figure}
%
\subsection{Extension to higher order}
This subsection is devoted to a numerical study of an implicit--explicit BDF--Galerkin scheme of third order in time. Following the construction of Section~\ref{Sec.Galerkin}, we combine the standard BDF-$3$ operator 
\[
    D_\tau^* U_h^n 
    = \frac{1}{6\tau}\, \Big(11\, U_h^n - 18\,U_h^{n-1} + 9\,U_h^{n-2} - 2\,U_h^{n-3}\Big),\qquad n\ge3
\]
with a third-order extrapolation of the conductivity, namely 
\[
    \sigma^*(U_h^{n-1})
    \coloneqq 3\,\sigma(U_h^{n-1}) - 3\,\sigma(U_h^{n-2}) + \sigma(U_h^{n-3}).
\]
To initialize the method for $n \leq 2$, we use the exact solution, ensuring that the initial error does not affect the third-order accuracy. Alternatively, one may follow~\cite{gao2016} and compute the starting values by a lower-order initialization procedure involving an implicit--explicit Euler step and Crank--Nicolson-type corrections. This then yields approximations with temporal error of order $O(\tau^3)$. Such a construction preserves the global third-order accuracy of the scheme. 
For $n=3,\ldots,N$, the fully discrete third-order BDF--Galerkin FEM then takes the form
\begin{subequations}
\label{eq:BDF3}
\begin{align}
    (D_{\tau}^*U_h^{n},\xi_u) + (\nabla U_h^{n},\nabla \xi_u) 
    &= \big(\sigma^*(U_h^{n-1})|\nabla\Phi_{h}^{n}|^2,\xi_u\big),\\
    \big(\sigma^*(U_h^{n-1})\nabla\Phi_h^{n},\nabla\xi_\phi\big) 
    &= 0 
\end{align}
\end{subequations}
for all $\xi_u,\, \xi_\phi\in \Vo$. Note that the boundary conditions remain unchanged. 

Since we expect an $L^2$-error of order $\calO(h^2+\tau^3)$, we choose $\tau = h^{2/3}$ for the time step size. Indeed, this leads to a convergence rate of order two. Also for the other error components, we obtain the expected optimal rates, i.e., order one in the $H^1$-norm and order two in the supercloseness and superconvergence estimates. 
Finally, to verify the expected third-order convergence in time, we fix a sufficiently fine spatial mesh and reduce the time step size. The resulting $L^2$-errors at time $t=T$ are plotted in Figure~\ref{Fig:BDF3:errors:tau}, indicating third-order convergence in time.
\begin{figure}
\centering
    \begin{tikzpicture}

\begin{axis}[%
width=2.8in,
height=1.7in,
scale only axis,
xmode=log,
xlabel={time step size $\tau$},
ylabel={\small$\|(U_h^N,\Phi_h^N)-(u^N,\phi^N)\|_0$},
xmin=0.0012,
xmax=0.12,
xminorticks=true,
ymode=log,
ymin=0.0000005,
ymax=0.0005,
yminorticks=true,
axis background/.style={fill=white},
legend style={at={(1.48,1.0)}, legend cell align=left, align=left, draw=white!15!black},
legend columns=1
]

\addplot [color=mycolor2, mark=pentagon*, line width=1.5pt]
  table[row sep=crcr]{%
0.1         0.00027232\\
0.05       0.00026618\\
0.025      0.00026773\\
0.0125     0.00026796\\
0.00625    0.00026799\\
0.003125   0.00026799\\
0.0015625  0.00026799\\
};
\addlegendentry{$h=0.0625$} 

\addplot [color=mycolor3, mark=triangle*, line width=1.5pt, mark size=2.8pt]
table[row sep=crcr]{%
0.1         0.00007373\\
0.05       0.00006588\\
0.025      0.00006643\\
0.0125     0.00006665\\
0.00625    0.00006668\\
0.003125   0.00006668\\
0.0015625  0.00006668\\
};
\addlegendentry{$h=0.0312$} 

\addplot [color=mycolor4, mark=diamond*, line width=1.5pt, mark size=2.8pt]
table[row sep=crcr]{%
0.1        0.00002474\\
0.05      0.00001649\\
0.025     0.00001647\\
0.0125    0.00001662\\
0.00625   0.00001665\\
0.003125  0.00001665\\
0.0015625 0.00001665\\
};
\addlegendentry{$h=0.0156$} 

\addplot [color=mycolor5, mark=*, line width=1.5pt]
table[row sep=crcr]{%
0.1        0.00002534\\
0.05       0.00000420\\
0.025      0.00000409\\
0.0125     0.00000413\\  
0.00625    0.00000416\\  
0.003125   0.00000416\\  
0.0015625  0.00000416\\  
};
\addlegendentry{$h=0.0078$} 

\addplot [color=mycolor6, mark=square*, line width=1.5pt]
table[row sep=crcr]{%
0.1        0.00002834\\
0.05       0.00000247\\
0.025      0.00000102\\
0.0125     0.00000102\\  
0.00625    0.00000104\\  
0.003125   0.00000104\\  
0.0015625  0.00000104\\  
};
\addlegendentry{$h=0.0039$} 

\addplot [color=black, dashed, line width=1.0pt]
table[row sep=crcr]{%
0.1       0.00001\\      
0.05      0.00000125\\   
0.025     0.00000015625\\ 
0.0125    0.00000001953125\\ 
};
\addlegendentry{order $3$} 
\end{axis}
\end{tikzpicture}
    \caption{Convergence history of the $L^2$-errors in $u$ and $\phi$ for the third-order BDF--Galerkin FEM~\eqref{eq:BDF3} for fixed spatial meshes.}
    \label{Fig:BDF3:errors:tau}
\end{figure}
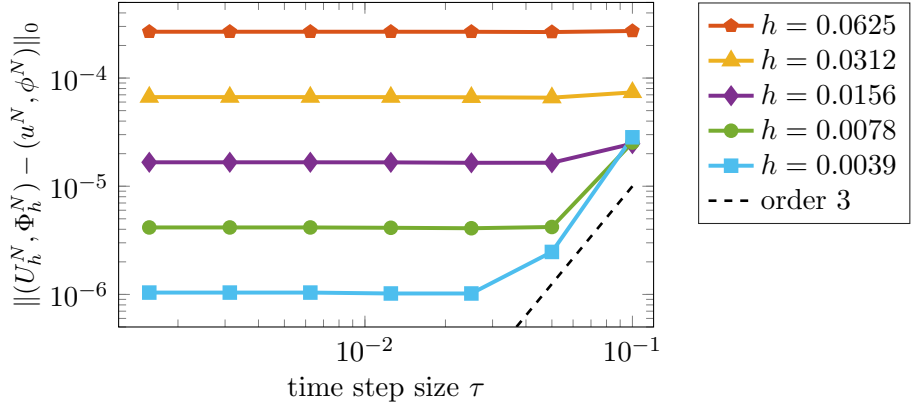
%
\subsection{Comparison to related methods}
In this subsection, we highlight the advantages of our method and present a comparison with related numerical schemes. 
First, we consider the second-order BDF--Galerkin FEM investigated in~\cite{gao2016}, which reads 
\begin{subequations}
\label{eq:BDF2:gao}
    \begin{align}
        (D_{\tau}U_h^{n},\xi_u) + (\nabla U_h^{n},\nabla \xi_u)& =  \left(2\sigma(U_h^{n-1})|\nabla\Phi_{h}^{n-1}|^2-\sigma(U_h^{n-2})|\nabla\Phi_{h}^{n-2}|^2,\xi_u\right), \\
        \left(\sigma(U_h^{n})\nabla\Phi_h^{n},\nabla\xi_\phi\right)& =0
\end{align}
\end{subequations}
for all $\xi_u,\, \xi_\phi\in \Vo$. Note that, compared to scheme~\eqref{eq:BDF:Galerkin}, different extrapolation formulae are used in both equations. To compare the results with the previous experiments, we choose again  $\tau = h$. As shown in Table~\ref{tab:u:gao}, this reduces the convergence order in the interpolation-type errors.
\begin{table}
\caption{Numerical errors and convergence orders in $u$ and $\phi$ at time $t = T$ using the BDF--Galerkin FEM~\eqref{eq:BDF2:gao} with $\tau=h$.}
\label{tab:u:gao}
\centering
\begin{tabular}{c|cccc}
\toprule
\text{mesh}
& $32\times 32$
& $64\times 64$
& $128\times 128$
& $256\times 256$ \\
\midrule
$\|U_h^N-I_h u^N\|_1$
& $2.17\cdot10^{-2}$
& $1.04\cdot10^{-2}$
& $5.25\cdot10^{-3}$
& $2.62\cdot10^{-3}$ \\
order
& --
& $1.05$
& $0.99$
& $1.00$ \\
\midrule
$\|\Phi_h^N-I_h\phi^N\|_1$
& $2.70\cdot10^{-2}$
& $1.34\cdot10^{-2}$
& $6.69\cdot10^{-3}$
& $3.35\cdot10^{-3}$ \\
order
& --
& $1.01$
& $1.00$
& $1.00$ \\
\bottomrule
\end{tabular}
\end{table}

Second, we investigate the influence of the extrapolation term and consider the second-order BDF--Galerkin scheme~\eqref{eq:BDF:Galerkin} but with a first-order extrapolation. Likewise, one may interprete the scheme as the one in~\cite{yang2023}, replacing the backward Euler operator with the second-order BDF operator. More preceisely, we study
\begin{subequations}
\label{eq:BDF2:nm1}
    \begin{align}
        (D_{\tau}U_h^{n},\xi_u) + (\nabla U_h^{n},\nabla \xi_u)& =  \left(\sigma(U_h^{n-1})|\nabla\Phi_{h}^{n}|^2,\xi_u\right), \\
        \left(\sigma(U_h^{n-1})\nabla\Phi_h^{n},\nabla\xi_\phi\right)& =0
\end{align}
\end{subequations}
for all $\xi_u,\, \xi_\phi\in \Vo$. Using again $\tau = h$, not only the interpolation-type error $\| \Phi_h^N - I_h \phi^N \|_1$ but also the $L^2$-error $\|\Phi_h^N - \phi^N\|_0$ shows a reduced convergence order of one.
%
%
\subsection{Further experiments}
To demonstrate the validity of the theoretical analysis for more realistic thermistor problems, we study a number of additional experiments. In the following, we only discuss the numerical results for $u$, since the approximations of $\phi$ exhibit the corresponding convergence behavior. 
%
\subsubsection{Three-dimensional problem}
Although not covered by the theoretical analysis, we also consider a three-dimensional numerical experiment. For this, we consider problem~\eqref{eq:thermistor:f} on the unit cube $\Omega=(0,1)^3$ with right-hand sides such that the exact solution is given by 
\[
    u(x,y,z,t) 
    = e^{-2t}\sin(\pi x)\sin(\pi y)\sin(\pi z),\qquad 
    \phi(x,y,z,t) 
    = 1 + \sin(x+y+z+t).
\]
We partition the domain into cubes with grid resolutions of $M\times M\times M$ with $M = 4, 8, 16 $. The corresponding mesh size is given by $h = \sqrt{3}/M$ and we set $\tau = h$ for the time step size. The convergence, supercloseness, and superconvergence results of the proposed BDF--Galerkin scheme~\eqref{eq:BDF:Galerkin} for $u$ at time $t = T$ are reported in Table~\ref{tab:u_convergence_3d}. One can see that the expected convergence behavior is retained also in three space dimensions. 
\begin{table}
\centering
\caption{Numerical errors for $u$ at final time $t=T$ using the second-order BDF--Galerkin FEM~\eqref{eq:BDF:Galerkin} for the three-dimensional problem.}
\label{tab:u_convergence_3d}
\begin{tabular}{c|ccc}
\toprule
\text{mesh}  & $4\times4\times4$ & $8\times8\times8$ & $16\times16\times16$ \\
\midrule
$\|U_h^n - u^n\|_0$ 
& $2.4449\cdot 10^{-3}$ 
& $5.5167\cdot 10^{-4}$ 
& $1.2898\cdot 10^{-4}$ \\
order 
& -- 
& $2.15$ 
& $2.10$ \\
\addlinespace
$\|U_h^n - u^n\|_1$ 
& $7.2627\cdot 10^{-2}$ 
& $3.2521\cdot 10^{-2}$ 
& $1.7377\cdot 10^{-2}$ \\
order 
& -- 
& $1.16$ 
& $0.90$ \\
\addlinespace
$\|U_h^n - I_h u^n \|_1$ 
& $2.5875\cdot 10^{-2}$ 
& $5.1381\cdot 10^{-3}$ 
& $1.0929\cdot 10^{-3}$ \\
order 
& -- 
& $2.33$ 
& $2.23$ \\
\addlinespace
$\|I_{2h} U_h^n - u^n \|_1$ 
& $7.8195\cdot 10^{-2}$ 
& $1.1407\cdot 10^{-2}$ 
& $3.7514\cdot 10^{-3}$ \\
order 
& -- 
& $2.78$ 
& $1.60$ \\
\bottomrule
\end{tabular}
\end{table}
%
\subsubsection{Different conductivities}
In the upcoming examples, we consider the original system~\eqref{eq:thermistor} without artificial forcing. To evaluate the spatial convergence behavior, a reference solution is computed on a fine mesh ($M=256$, $\tau = 2^{-9}$). 
The computations are then performed on uniform rectangular meshes with mesh size $h=\sqrt{2}/M$, $M=16, 32, 64$, and time step size $\tau=h$. 

We investigate the numerical performance for different representative choices of the conductivity function $\sigma$. For this, we consider a strongly decreasing 
\begin{equation}\label{eq:sigma:decreasing}
\sigma(u) = \frac{2}{1+e^u}+0.1    
\end{equation}
as well as an increasing behavior, 
\begin{equation}\label{eq:sigma:increasing}
\sigma(u)=0.5+0.5\tanh(u).
\end{equation}
Such nonlinear conductivities are admissible under the classical thermistor theory, provided they are continuous, uniformly positive, and bounded~\cite{AntC94}. Similar choices are standard in both analytical and numerical studies of thermistor problems~\cite{chen1993,durojaye2019}. Convergence rates in the $L^2$- and $H^1$-norms, as well as superconvergence-related quantities, are evaluated in order to assess whether the theoretical error estimates suffer from different nonlinear conductivity laws. The resulting errors for the strongly decreasing conductivity~\eqref{eq:sigma:decreasing} are listed in Table~\ref{tab:u_convergence_sigmaB}, indicating that the proposed method preserves its optimal convergence and superconvergence properties. Very similar results are obtained for the increasing conductivity~\eqref{eq:sigma:increasing} such that we omit the actual convergence results. 
\begin{table}
\centering
\caption{ Numerical errors for the temperature $u$ at final time $t=T$ for the thermistor problem with strongly decreasing conductivity \eqref{eq:sigma:decreasing}, computed by the second-order BDF--Galerkin FEM~\eqref{eq:BDF:Galerkin} with $\tau=h$.}
\label{tab:u_convergence_sigmaB}
\begin{tabular}{c|ccc}
\toprule
\text{mesh} & $16\times16$ & $32\times32$ & $64\times64$ \\
\midrule
$\|U_h^n - u^n\|_0$
& $2.710\times 10^{-4}$
& $6.812\times 10^{-5}$
& $1.720\times 10^{-5}$ \\
order
& -- 
& $1.99$
& $1.99$ \\
\addlinespace
$\|U_h^n - u^n\|_1$
& $2.234\times 10^{-2}$
& $1.100\times 10^{-2}$
& $4.855\times 10^{-3}$ \\
order
& --
& $1.02$
& $1.18$ \\
\addlinespace
$\|U_h^n - I_h u^n \|_1$
& $2.163\times 10^{-3}$
& $6.389\times 10^{-4}$
& $1.812\times 10^{-4}$ \\
order
& --
& $1.76$
& $1.82$ \\
\addlinespace
$\|I_{2h} U_h^n - u^n \|_1$
& $9.252\times 10^{-4}$
& $2.401\times 10^{-4}$
& $6.031\times 10^{-5}$ \\
order
& --
& $1.95$
& $1.99$ \\
\bottomrule
\end{tabular}
\end{table}
%
\subsubsection{Non-convex domains}
Next, we consider system~\eqref{eq:thermistor} on the non-convex L-shape domain 
\[
 \Omega = (0,1)^2\, /\, ([0.5,1]\times[0.5,1]),
\]
which is known to cause reduced regularity in elliptic and parabolic problems. In this case, the boundary conditions are selected to generate nontrivial electric currents, i.e., $g(x,y)= x$, while no artificial source terms are added to the governing equations. This example allows us to examine the influence of geometric singularities on the numerical solution.  
The resulting convergence behavior for $u$, showing degraded rates in all measures, is given in Table~\ref{tabL:u_convergence}. 
\begin{table}
\centering
\caption{Numerical errors for $u$ at final time $t=T$ on the L-shape, computed with the second-order BDF--Galerkin FEM~\eqref{eq:BDF:Galerkin} and {$\tau = h$}.}
\label{tabL:u_convergence}
\begin{tabular}{c|ccc}
\toprule
\text{mesh}  & $16\times 16$ & $32\times 32$ & $64\times 64$ \\ 
\midrule
$\|U_h^n - u^n\|_0$ 
& $7.709\cdot 10^{-4}$ 
& $2.199\cdot 10^{-4}$ 
& $6.657\cdot 10^{-5}$ \\
order 
& -- 
& $1.81$ 
& $1.72$ \\
\addlinespace
$\|U_h^n - u^n\|_1$ 
& $4.418\cdot 10^{-2}$ 
& $2.416\cdot 10^{-2}$ 
& $1.358\cdot 10^{-2}$ \\
order 
& -- 
& $0.87$ 
& $0.83$ \\
\addlinespace
$\|U_h^n - I_h u^n \|_1$ 
& $1.194\cdot 10^{-4}$ 
& $5.204\cdot 10^{-5}$ 
& $2.033\cdot 10^{-5}$ \\
order 
& -- 
& $1.20$ 
& $1.36$ \\
\addlinespace
$\|I_{2h} U_h^n - u^n \|_1$ 
& $7.763\cdot 10^{-4}$ 
& $2.225\cdot 10^{-4}$ 
& $6.533\cdot 10^{-5}$ \\
order 
& -- 
& $1.80$ 
& $1.77$ \\
\bottomrule
\end{tabular}
\end{table}
%
\subsubsection{Non-smooth boundary data}
Finally, we investigate the performance of the method under reduced regularity due to non-smooth boundary data. For this, the temperature satisfies homogeneous Dirichlet conditions as before, whereas the electric potential is equipped with discontinuous boundary conditions of the form $g(x,y)=0$ for $x<0.5$ and $g(x,y) =1$ for $x\ge 0.5$. The corresponding errors are reported in Table~\ref{tab:u_nonsmooth}. 
As expected, the discontinuity in the boundary data leads to reduced regularity of the solution and a dramatic drop of the convergence rates. While the $L^2$-errors still show convergence of order one, convergence in the $H^1$-norm is noticeably slower.   
\begin{table}
\centering
\caption{Numerical errors for $u$ at final time $t=T$ with non-smooth boundary conditions, computed using the second-order BDF--Galerkin FEM~\eqref{eq:BDF:Galerkin} and $\tau=h$.}
\label{tab:u_nonsmooth}
\begin{tabular}{c|ccc}
\toprule
\text{mesh} &  $16\times16$ & $32\times32$ & $64\times64$ \\ 
\midrule
$\|U_h^n-u^n\|_0$ 
& $3.438\times10^{-2}$ 
& $1.767\times10^{-2}$ 
& $8.251\times10^{-3}$ \\
order
& -- 
& $0.96$ 
& $1.10$ \\
\addlinespace
$\|U_h^n-u^n\|_1$ 
& $1.596\times10^{0}$ 
& $1.360\times10^{0}$ 
& $1.054\times10^{0}$ \\
order
& -- 
& $0.23$ 
& $0.37$ \\
\addlinespace
$\|U_h^n-I_hu^n\|_1$ 
& $1.582\times10^{0}$ 
& $1.334\times10^{0}$ 
& $1.012\times10^{0}$ \\
order
& -- 
& $0.25$ 
& $0.40$ \\
\addlinespace
$\|I_{2h}U_h^n-u^n\|_1$ 
& $2.719\times10^{-1}$ 
& $2.492\times10^{-1}$ 
& $2.095\times10^{-1}$ \\
order
& -- 
& $0.13$ 
& $0.25$ \\
\bottomrule
\end{tabular}
\end{table} 
%
%
\section{Concluding Remarks}\label{Sec:Conclusion}
We have proposed and analyzed an implicit--explicit BDF--Galerkin method for the time-dependent nonlinear thermistor problem. The resulting decoupling of the equations improves the computational efficiency compared to fully implicit methods. For this scheme, we have established supercloseness and superconvergence results without any restriction on the time step size. Numerical experiments verify the theoretical results and demonstrate that the method reaches the expected accuracy and is indeed unconditionally stable.

While we have focused on BDF of second order, the analytical techniques developed here may be extended to higher-order methods with appropriate extrapolation terms. In particular, this includes the temporal--spatial error splitting technique and the telescope formula of the temporal discretization operator. Also the convergence analysis of the three-dimensional case is a promising field of future research. 
Moreover, the techniques presented in this work may be extended to derive implicit--explicit methods for other nonlinear systems.  
%
%
%
%
\bibliographystyle{alpha} 
\bibliography{bib_Joule}

@book{adams2003,
  title={Sobolev spaces},
  author={Adams, R. A. and Fournier, J. J. F.},
  year={2003},
  publisher={Elsevier, Amsterdam}
}

@article{durojaye2019,
  title={The thermistor problem with hyperbolic electrical conductivity},
  author={Durojaye, M. O. and Agee, J. T.},
  journal={Asian Res. J. Math.},
  volume={14},
  number={2},
  pages={1--12},
  year={2019}
}

@article{chen1993,
  title={The thermistor problem for conductivity which vanishes at large temperature},
  author={Chen, X. and Friedman, A.},
  journal={Quar. Appl. Math.},
  volume={51},
  number={1},
  pages={101--115},
  year={1993}
}

@book{thomee2007,
  title={Galerkin finite element methods for parabolic problems},
  author={Thom{\'e}e, V.},
  year={2007},
  publisher={Springer Science \& Business Media, Berlin}
}

@article{gao2014,
  title={Optimal error analysis of {Galerkin} {FEM}s for nonlinear {Joule} heating equations},
  author={Gao, H.},
  journal={J. Sci. Comput.},
  volume={58},
  number={3},
  pages={627--647},
  year={2014},
  publisher={Springer}
}

@article{gao2016,
  title={Unconditional optimal error estimates of {BDF}--{Galerkin} {FEM}s for nonlinear thermistor equations},
  author={Gao, H.},
  journal={J. Sci. Comput.},
  volume={66},
  number={2},
  pages={504--527},
  year={2016},
  publisher={Springer}
}

@article{dupont1974,
  title={Three-level {Galerkin} methods for parabolic equations},
  author={Dupont, T. and Fairweather, G. and Johnson, J. P.},
  journal = {SIAM J. Numer. Anal.},
  volume={11},
  number={2},
  pages={392--410},
  year={1974},
  publisher={SIAM}
}

@article{nirenberg1966,
  title={An extended interpolation inequality},
  author={Nirenberg, L.},
 journal = {Ann. Sc. Norm. Super. Pisa Cl. Sci.},
  volume={20},
  number={4},
  pages={733--737},
  year={1966}
}

@article{yang2023,
  title={Unconditionally superconvergent error estimates of a linearized {Galerkin} finite element method for the nonlinear thermistor problem},
  author={Yang, H. and Shi, D.},
  journal = {Adv. Comput. Math.},
  volume={49},
  number={3},
  pages={33},
  year={2023},
  publisher={Springer}
}

@book{brenner2008,
  title={The mathematical theory of finite element methods},
  author={Brenner, S. C. and Scott, L. R.},
  year={2008},
  publisher={New York, NY: Springer New York}
}

@article{elliott1995,
  title={A finite element model for the time-dependent {Joule} heating problem},
  author={Elliott, C. M. and Larsson, S.},
  journal = {Math. Comp.},
  volume={64},
  number={212},
  pages={1433--1453},
  year={1995}
}

@article{zhao1994,
  title={Convergence Analysis of Finite Element Method for the Nonstationary Thermistor Problem},
  author={Zhao, W.},
  journal = {J. Shandong Univ.},
  volume={29},
  pages={361--367},
  year={1994}
}

@article{li2012,
  title={Error Analysis of Linearized Semi-Implicit {Galerkin} Finite Element Methods for Nonlinear Parabolic Equations},
  author={Li, B. and Sun, W.},
  journal = {Int. J. Numer. Anal. Model.},
  volume={10},
  pages={622--633},
  year={2013}
}

@article{li2014,
  title={Unconditionally optimal error estimates of a {Crank}--{Nicolson} {Galerkin} method for the nonlinear thermistor equations},
  author={Li, B. and Gao, H. and Sun, W.},
  journal = {SIAM J. Numer. Anal.},
  volume={52},
  number={2},
  pages={933--954},
  year={2014},
  publisher={SIAM}
}

@article{heywood1990,
  title={Finite-element approximation of the nonstationary {Navier--Stokes} problem. {Part IV}: error analysis for second-order time discretization},
  author={Heywood, J. G. and Rannacher, R.},
  journal = {SIAM J. Numer. Anal.},
  volume={27},
  number={2},
  pages={353--384},
  year={1990},
  publisher={SIAM}
}

@book{chen1998,
  title={Second order elliptic equations and elliptic systems},
  author={Chen, Y.-Z. and Wu, L.-C.},
  volume={174},
  year={1998},
  publisher={American Mathematical Society, Providence}
}

@article{liu2013,
  title={Simple and efficient {ALE} methods with provable temporal accuracy up to fifth order for the {Stokes} equations on time varying domains},
  author={Liu, J.},
  journal = {SIAM J. Numer. Anal.},
  volume={51},
  number={2},
  pages={743--772},
  year={2013},
  publisher={SIAM}
}

@book{lin2007,
  title={Finite Element Methods: Accuracy and Improvement},
  author={Lin, Q. and Lin, J.},
  volume={1},
  year={2007},
  publisher={Elsevier, Amsterdam}
}

@article{allegretto1992,
  title={Existence of Solutions for the Time-Dependent Thermistor Equations},
  author={Allegretto, W. and Xie, H.},
  journal = {IMA J. Appl. Math.},
  volume={48},
  number={3},
  pages={271--281},
  year={1992},
  publisher={Oxford University Press}
}

@article{allegretto2002,
  title={Existence and Long Time Behaviour of Solutions to Obstacle Thermistor Equations},
  author={Allegretto, W. and Lin, Y. and Ma, S.},
  journal = {Discrete Contin. Dyn. Syst.},
  volume={8},
  number={3},
  pages={757--780},
  year={2002},
  publisher={Discrete and Continuous Dynamical Systems}
}

@article{cimatti1992,
  title={Existence of weak solutions for the nonstationary problem of the {Joule} heating of a conductor},
  author={Cimatti, G.},
  journal = {Ann. Mat. Pura Appl.},
  volume={162},
  pages={33--42},
  year={1992},
  publisher={Springer}
}

@article{yuan1994a,
  title={Regularity of solutions of the thermistor problem},
  author={Yuan, G.},
  journal={Appl. Anal.},
  volume={53},
  number={3--4},
  pages={149--156},
  year={1994},
  publisher={Taylor \& Francis}
}

@article{akrivis2005,
  title={Linearly implicit finite element methods for the time-dependent {Joule} heating problem},
  author={Akrivis, G. and Larsson, S.},
  journal = {BIT Numer. Math.},
  volume={45},
  pages={429--442},
  year={2005},
  publisher={Springer}
}

@article{shi2018,
  title={Superconvergent estimates of conforming finite element method for nonlinear time-dependent {Joule} heating equations},
  author={Shi, D. and Yang, H.},
  journal = {Numer. Methods Partial Differ. Equ.},
  volume={34},
  number={1},
  pages={336--356},
  year={2018},
  publisher={Wiley Online Library}
}

@article{shi2019,
  title={Superconvergence analysis of nonconforming {FEM} for nonlinear time-dependent thermistor problem},
  author={Shi, D. and Yang, H.},
  journal={Appl. Math. Comput.},
  volume={347},
  pages={210--224},
  year={2019},
  publisher={Elsevier}
}

@article{JenMP20,
    author = {Jensen, M. and Målqvist, A. and Persson, A.},
    title = {Finite element convergence for the time-dependent {J}oule heating problem with mixed boundary conditions},
    journal = {IMA J. Numer. Anal.},
    fjournal = {IMA Journal of Numerical Analysis},
    volume = {42},
    number = {1},
    pages = {199--228},
    year = {2020},
    issn = {0272-4979},
    doi = {10.1093/imanum/draa068},
    url = {https://doi.org/10.1093/imanum/draa068},
    eprint = {https://academic.oup.com/imajna/article-pdf/42/1/199/42098183/draa068.pdf},
}

@article{AntC94,
  author = {Antontsev, S. N. and Chipot, M.},
  title = {The Thermistor Problem: Existence, Smoothness Uniqueness, Blowup},
  journal = {SIAM J. Math. Anal.},
  fjournal = {SIAM Journal on Mathematical Analysis},
  volume = {25},
  number = {4},
  pages = {1128--1156},
  year = {1994},
  doi = {10.1137/S0036141092233482},
  URL = {https://doi.org/10.1137/S0036141092233482},
}

@article {AltM22,
	AUTHOR = {Altmann, R. and Maier, R.},
	TITLE = {A decoupling and linearizing discretization for weakly coupled poroelasticity with nonlinear permeability},
	JOURNAL = {SIAM J. Sci. Comput.},
	FJOURNAL = {SIAM Journal on Scientific Computing},
	VOLUME = {44},
	YEAR = {2022},
	NUMBER = {3},
	PAGES = {B457--B478},
}

\end{document}